\documentclass[12pt]{article}
\usepackage[margin=1in]{geometry}

\usepackage{tikz-cd}
\usepackage{tikz}
\usepackage{adjustbox}
\usepackage[algo2e,ruled,vlined]{algorithm2e}
\usepackage[normalem]{ulem}
\usepackage{multirow}
\usepackage{amsmath}
\usepackage{amssymb}
\usepackage{amsthm}
\usepackage{authblk}
\usepackage{graphicx}

\long\def\remove#1{}

\newcommand{\tamal}[1]{\textcolor{darkblue}{{Tamal says:#1}}}

\newcommand{\cancel}[1]{}

\DeclareMathOperator{\inv}{\sf Inv}
\DeclareMathOperator{\esol}{\sf eSol}
\DeclareMathOperator{\cl}{\sf cl}
\DeclareMathOperator{\mo}{\sf mo}
\DeclareMathOperator{\pf}{\sf pf}
\DeclareMathOperator{\pb}{\sf pb}

\newtheorem{theorem}{\sffamily Theorem}

\newtheorem{corollary}[theorem]{\sffamily Corollary}
\newtheorem{proposition}[theorem]{\sffamily Proposition}
\newtheorem{definition}[theorem]{\sffamily Definition}

\newif\ifpaper
\papertrue

\title{Persistence of the Conley Index in Combinatorial Dynamical Systems}

\begin{document}

\author[1]{Tamal K. Dey\thanks{dey.8@osu.edu} }

\author[2]{Marian Mrozek\thanks{marian.mrozek@uj.edu.pl}}

\author[1]{Ryan Slechta\thanks{slechta.3@osu.edu}}

\affil[1]{Department of Computer Science and Engineering,
The Ohio State University, Columbus, USA}
\affil[2]{Division of Computational Mathematics, Faculty of Mathematics and Computer Science, Jagiellonian University, Krak\'{o}w, Poland}
\date{}

\maketitle

\setcounter{page}{1}
\begin{abstract}
A combinatorial framework for dynamical systems provides an avenue for connecting classical dynamics with data-oriented, algorithmic methods. Combinatorial vector fields introduced by Forman \cite{Forman1998b,Forman1998a} and their recent generalization to multivector fields \cite{Mr2017} have provided a starting point for building such a connection. In this work, we strengthen this relationship by placing the Conley index in the persistent homology setting. Conley indices are homological features associated with so-called isolated invariant sets, so a change in the Conley index is a response to perturbation in an underlying multivector field. We show how one can use zigzag persistence to summarize changes to the Conley index, and we develop techniques to capture such changes in the presence of noise. We conclude by developing an algorithm to ``track'' features in a changing multivector field.
\end{abstract}

\section{Introduction}
\label{sec:intro}
At the end of the 19th century, scientists became aware that the very fruitful theory of differential equations cannot provide a description of the asymptotic behavior of solutions in situations when no analytic formulas for solutions are available. This observation affected Poincar\'e's study on the stability of our celestial system \cite{Poinc1890} and prompted him to use the methods of dynamical systems theory. The fundamental observation of the theory is that solutions limit in invariant sets. Examples of invariant sets include stationary solutions, periodic orbits, connecting orbits, and many more complicated sets such as chaotic invariant sets discovered in the second half of the 20th century \cite{Lo1963}. 
Today, the Conley index \cite{Co78,MiMr2002} is among the most fundamental topological descriptors that are used for analyzing invariant sets. The Conley index is defined for {\em isolated invariant sets} which are maximal invariant sets in some neighborhood. It characterizes whether isolated invariant sets are attracting, repelling or saddle-like. It is used to detect stationary points, periodic solutions and connections between them. Moreover, it provides methods to detect and characterize different chaotic invariant sets. In particular, it was used to prove that the system discovered by Lorenz \cite{Lo1963} actually contains a chaotic invariant set \cite{MiMr1995}. The technique of multivalued maps used in this proof
may be adapted to dynamical systems known only from finite samples \cite{MMRS1999}. Unfortunately, unlike the case when an analytic description 
of the dynamical system is available, the approach proposed in \cite{MMRS1999} lacks a validation method. This restricts possible applications in today's data-driven world. 
In order to use topological persistence as a validation tool, we need an analog of dynamical systems for discrete data. 
In the case of a dynamical system with continuous time, the idea comes from the fundamental work of R. Forman on discrete Morse theory \cite{Forman1998a} and combinatorial vector fields \cite{Forman1998b}. This notion of a combinatorial vector field was recently generalized to that of a combinatorial multivector field \cite{Mr2017}. Since then, the Conley index has been constructed in the setting of combinatorial multivector fields \cite{LKMW19,Mr2017}. The aim of this research is to incorporate the ideas of topological persistence into the study of the Conley index.

\begin{figure}[htbp]
\centering
\begin{tabular}{ccc}
  \includegraphics[height=35mm]{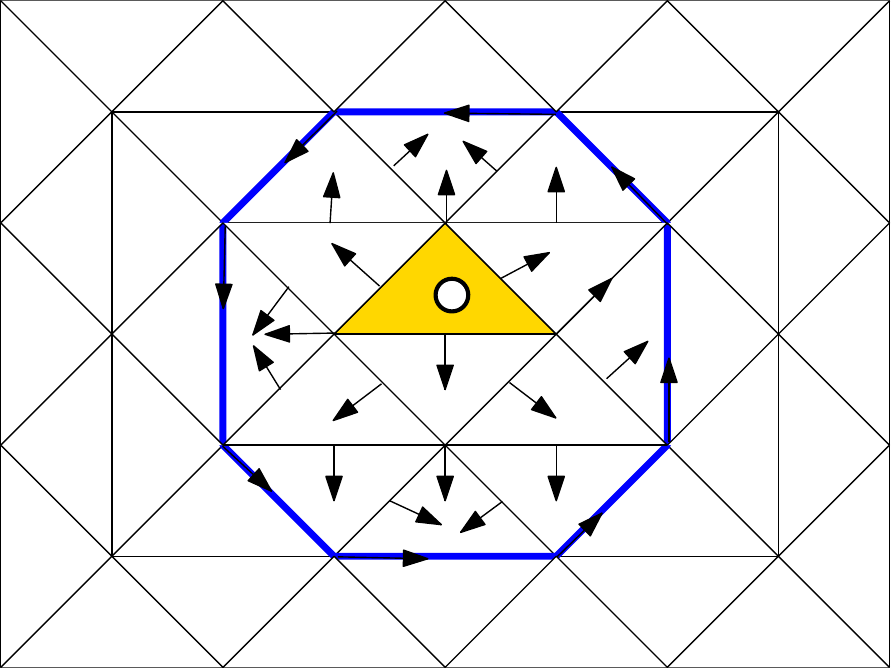}&
  \includegraphics[height=35mm]{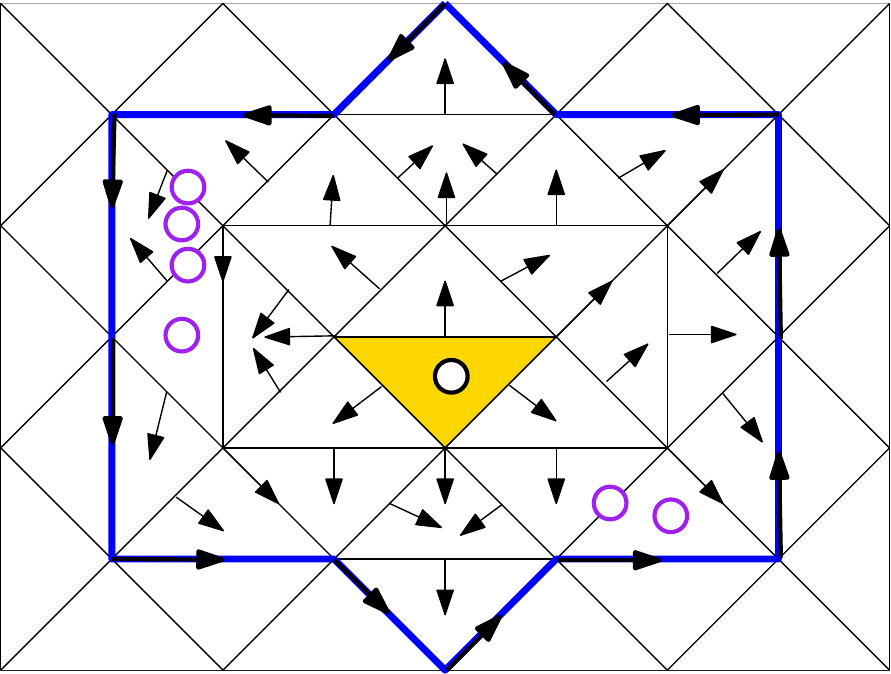}&
  \includegraphics[height=35mm]{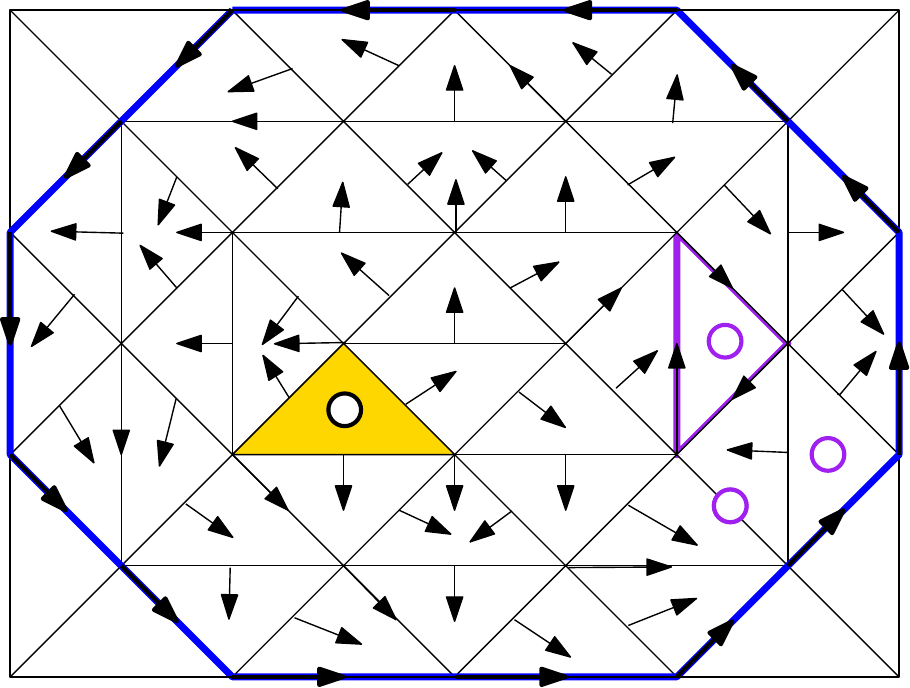}%
\end{tabular}
\caption{
Three multivector fields. In each field, there is a periodic attractor in blue. Such an attractor is an example of an invariant set. The reader will notice that all flows which enter the periodic attractor can ultimately be traced back to simplices marked with circles. These simplices are individually invariant sets, and they correspond to the notion of fixed points. In each multivector field, the gold triangle corresponds to the notion of a repelling fixed point, while the  triangles and edges with magenta circles are spurious. Notice that in the third multivector field there is a spurious periodic attractor. However, despite spurious invariant sets, in all three multivector fields the predominant feature is a repelling triangle, from which most emanating flow terminates in the periodic attractor. We aim to develop a quantitative summary of this behavior.}
\label{intro-fig}
\end{figure} 

Given a simplicial complex $K$, Forman defined a combinatorial vector field $\mathcal{V}$ as a partition of $K$ into three sets $L\sqcup U\sqcup C$ where a bijective map $\mu: L \rightarrow U$ pairs a $p$-simplex $\sigma\in L$ with a $(p+1)$-simplex $\tau=\mu(\sigma)$. This pair can be thought of as a vector originating in $\sigma$ and terminating in $\tau$. Using these vectors, Forman defined a notion of flow for discrete vector fields called a $V$-path. These paths correspond to the classical notion of integral lines in smooth vector fields. Multivector fields proposed in~\cite{Mr2017} generalize this concept by allowing a vector to have multiple simplices (dubbed \textit{multivectors}) and more complicated dynamics. 

An extension to the idea of the $V$-path from Forman's theory is called a \textit{solution} for a multivector field $\mathcal{V}$. A solution in $\mathcal{V}$ is a possibly infinite sequence of simplices $\{\sigma_i\}$ such that $\sigma_{i+1}$ is either a face of $\sigma_i$ or in the same multivector as $\sigma_i$. Solutions may be doubly infinite (or bi-infinite), right infinite, left infinite, or finite. Solutions that are not doubly infinite are \emph{partial}.  
Bi-infinite solutions correspond to invariant sets in the combinatorial setting.

In Figure \ref{intro-fig}, one can see a sequence of multivector fields, each of which contains multiple isolated invariant sets. In principle, one would like to choose an isolated invariant set from each of the multivector fields and obtain a description of how the Conley index of these sets changes. Obtaining such a description is highly nontrivial, and it is the main contribution of this paper. Given a sequence of isolated invariant sets, we use the theory of zigzag persistence \cite{zigzag} to extract such a description. In \cite{DJKKLM19}, the authors studied the persistence of the Morse decomposition of multivector fields, but this is the first time that the Conley index has been placed in a persistence framework. We also provide schemes to automatically select isolated invariant sets and to limit the effects of noise on the persistence of the Conley index.
\section{Preliminaries}
\label{sec:prelim}
Throughout this paper, we will assume that the reader has a basic understanding of both point set and algebraic topology. In particular, we assume that the reader is well-versed in homology. For more information on these topics, we encourage the reader to consult \cite{hatcher,munkres}.

\subsection{Multivectors and Combinatorial Dynamics}

In this subsection, we briefly recall the fundamentals of multivector fields as established in \cite{DJKKLM19,LKMW19,Mr2017}. Let $K$ be a finite simplicial complex with face relation $\leq$, that is, $\sigma \leq \sigma'$ if and only if $\sigma$ is a face of $\sigma'$. Equivalently, $\sigma \leq \sigma'$ if $V(\sigma) \subseteq V(\sigma')$, where $V(\sigma)$ denotes the vertex set of $\sigma$. For a simplex $\sigma\in K$, we let $\cl(\sigma):=\{\tau\in K \,|\, \tau \leq \sigma\}$ and for a set $A\subseteq K$, we let $\cl(A):= \{\tau \leq \sigma \,|\,\sigma\in A\}$. We say that $A\subseteq K$ is \emph{closed} if $\cl (A)=A$.
The reader familiar with the Alexandrov topology \cite[Section 1.1]{Ba2011} will immediately notice that this notation and terminology is aligned with the topology induced on $K$ by the relation $\leq$.
\begin{definition}[Multivector, Multivector Field]
A subset $A\subseteq K$ is called a {\em multivector} if for all $\sigma,\sigma' \in A$, $\tau \in K$ satisfying $\sigma \leq \tau \leq \sigma'$, we have that $\tau \in A$. A {\em multivector field} over $K$ is a partition of $K$ into multivectors.
\end{definition}

 Every multivector is said to be either regular or critical. To define critical multivectors, we define the \textit{mouth} of a set as $\mo(A) := \cl(A) \setminus A$. The multivector $V$ is critical if the relative homology $H_p( \cl(V), \mo(V)) \neq 0$ in some dimension $p$. Otherwise, $V$ is regular. Simplices in critical multivectors are marked with circles in Figures \ref{intro-fig}, \ref{invariant-fig}, and \ref{mouthetc-fig}. Throughout this paper, all references to homology are references to simplicial homology. Note that $H_p\left(\cl(V), \mo(V)\right)$ is thus well defined because $\mo(S) \subseteq \cl(S) \subseteq K$. Intuitively, a multivector $V$ is regular if $\cl(V)$ can be collapsed onto $\mo(V)$. In Figure~\ref{invariant-fig}, the red triangle with its two edges is a critical multivector $V$ because  $H_1(\cl (V),\mo (V))$ is nontrivial. Similarly, the gold colored triangles (denoted $\tau$) and the green edge (denoted $\sigma$) are critical because $H_2(\cl (\tau),\partial \tau)$ and $H_1(\cl (\sigma),\partial \sigma)$ are nontrivial, where we use $\partial\sigma$ to denote the boundary of a simplex $\sigma$.

A multivector field over $K$ induces a notion of dynamics. For $\sigma \in K$, we denote the multivector containing $\sigma$ as $[\sigma]$. If the multivector field $\mathcal{V}$ is not clear from context, we will use the notation $[\sigma]_\mathcal{V}$. We now use a multivector field $\mathcal{V}$ on $K$ to define a multivalued map $F_\mathcal{V} \; : \; K \multimap K$. In particular, we let $F_\mathcal{V}(\sigma) := \cl(\sigma) \cup [\sigma]$. Such a multivalued map induces a notion of flow on $K$.
In the interest of brevity, for $a,b\in \mathbb{Z}$, we set $\mathbb{Z}_{[a,b]}=[a,b]\cap\mathbb{Z}$ and define $\mathbb{Z}_{(a,b]}$,  $\mathbb{Z}_{[a,b)}$, $\mathbb{Z}_{(a,b)}$ as expected. A \textit{path} from $\sigma$ to $\sigma'$ is a map $\rho \; : \; \mathbb{Z}_{[a,b]} \to K$, where $\rho(a) = \sigma$, $\rho(b) = \sigma'$, and for all $i \in \mathbb{Z}_{(a,b]}$, we have that $\rho(i) \in F_\mathcal{V}( \rho(i - 1) )$. Similarly, a \textit{solution} to a multivector field over $K$ is a map $\rho \; : \; \mathbb{Z} \to K$ where $\rho(i) \in F_\mathcal{V}(\rho(i-1))$. 
\begin{definition}[Essential Solution]
A solution $\rho \; : \; \mathbb{Z} \to K$ is an {\em essential solution} of multivector field $\mathcal{V}$ on $K$ if for each $i \in \mathbb{Z}$ where $[\rho(i)]$ is regular, there exists an $i^{-}, i^{+} \in \mathbb{Z}$ where $i^{-} < i < i^{+}$ and $[\rho(i^{-})] \neq [\rho(i)] \neq [\rho(i^{+})]$.
\end{definition}

For a set $A \subseteq K$,  
let $\esol(A)$ denote the set of essential solutions $\rho$ such that $\rho(\mathbb{Z}) \subseteq A$. If the relevant multivector field is not clear from context, we use the notation $\esol_\mathcal{V}(A)$. We define the \textit{invariant part} of $A$ as $\inv(A) = \{ \sigma \in A \; | \; \exists \rho \in \esol(A), \rho(0) = \sigma \}$. We say that $A$ is \textit{invariant} or an \textit{invariant set} if $\inv(A) = A$. If the multivector field is not clear from context, we use the notation $\inv_\mathcal{V}(A)$.

Solutions and invariant sets are defined in accordance to their counterparts in the classical 
setting. For more information on the classical counterparts of these concepts, see \cite{BhSz1967}. As in the classical setting, we have a notion of isolation for combinatorial invariant sets.
\begin{definition}[Isolated Invariant Set, Isolating Neighborhood]
An invariant set $A \subseteq N$, $N$ closed, is {\em isolated by $N$} if all paths $\rho \; : \; \mathbb{Z}_{[a,b]} \to N$ for which $\rho(a),\rho(b) \in A$ satisfy $\rho(\mathbb{Z}_{[a,b]}) \subseteq A$. The closed set $N$ is said to be an {\em isolating neighborhood} for $S$. 
\end{definition}

\begin{figure}[htbp]
\centering
\begin{tabular}{c}
  \frame{\includegraphics[height=35mm]{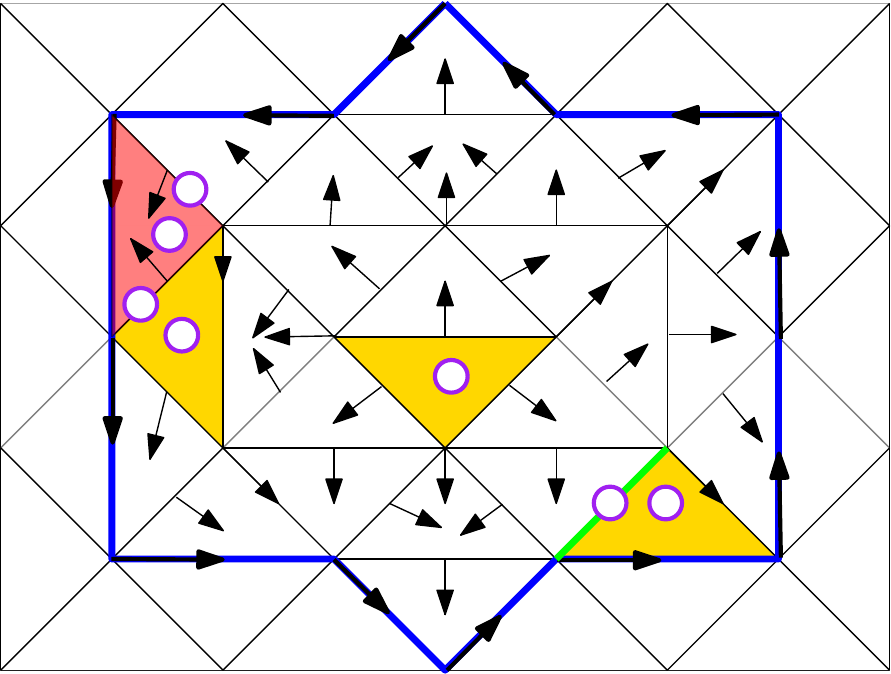}}
\end{tabular}
\caption{A multivector field with several invariant sets, isolated by the entire rectangle, $N$. Note that for each colored triangle $\sigma$, since $[\sigma]$ is critical, there is an essential solution $\rho \; : \; \mathbb{Z} \to N$ where $\rho(i) = \sigma$ for all $i$. Likewise for the green edge. Since the periodic attractor is composed of regular vectors, there is no such essential solution for any given simplex in the periodic attractor. However, by following the arrows in the periodic attractor we still get an essential solution. } 
\label{invariant-fig}
\end{figure} 

\subsection{Conley Indices}

The Conley index of an isolated invariant set is a topological invariant used to characterize features of dynamical systems~\cite{Co78,MiMr2002}. In both the classical and the combinatorial settings, the Conley index is determined by index pairs. 
\begin{definition}
Let $S$ be an isolated invariant set. The pair of closed sets $(P,E)$ subject to $E \subseteq P \subseteq K$ is an {\em index pair} for $S$ if all of the following hold:
\begin{enumerate}
\item $F_\mathcal{V}(E) \cap P \subseteq E$
\item $F_\mathcal{V}(P \setminus E) \subseteq P$ 
\item $S = \inv(P \setminus E)$
\end{enumerate}
\label{def:indpair}
\end{definition}
In addition, an index pair is said to be a \textit{saturated index pair} if $S = P \setminus E$. In Figure~\ref{mouthetc-fig}, the gold, critical triangle $\sigma$ is an isolated invariant set. The reader can easily verify that $\left(\cl(\sigma), \cl(\sigma)\setminus\{\sigma\}\right)$ is an index pair for $\sigma$.  In fact, this technique is a canonical way of picking an index pair for an isolated invariant set. This is formalized in the following proposition. 
\begin{proposition}{\cite[Proposition 4.3]{LKMW19}}
Let $S$ be an isolated invariant set. Then $(\cl(S), \mo(S))$ is a saturated index pair for $S$.
\label{prop:clomo}
\end{proposition}
However, there are several other natural ways to find index pairs. Figure~\ref{mouthetc-fig} shows another index pair for the same gold triangle $\sigma$. By letting $P := \cl(\sigma) \cup S_P$ and $E := \mo(\sigma) \cup S_E$, where $S_P$ and $S_E$ are the set of simplices reachable from paths originating in $\cl(S)$, $\mo(S)$ respectively, we obtain a much larger index pair. In Figure~\ref{mouthetc-fig}, $P$ is the set of all colored simplices, while $E$ is the set of all colored simplices which are not gold. 

In principle, it is important that the Conley index be independent of the choice of index pair. Fortunately, it is also known that the relative homology given by an index pair for an isolated invariant set $S$ is independent of the choice of index pair. 
\begin{theorem}{\cite[Theorem 4.15]{LKMW19}}
Let $(P_1,E_1)$ and $(P_2,E_2)$ be index pairs for the isolated invariant set $S$. Then $H_p(P_1,E_1) \cong H_p(P_2,E_2)$ for all $p$. 
\label{thm:ipiso}
\end{theorem}
The \textit{Conley Index}  of an isolated invariant set $S$ in dimension $p$ is then given by the relative
homology group $H_p(P,E)$ for any index pair of $S$ denoted $(P,E)$. 

\begin{figure}[htbp]
\centering
\begin{tabular}{c}
  \frame{\includegraphics[height=35mm]{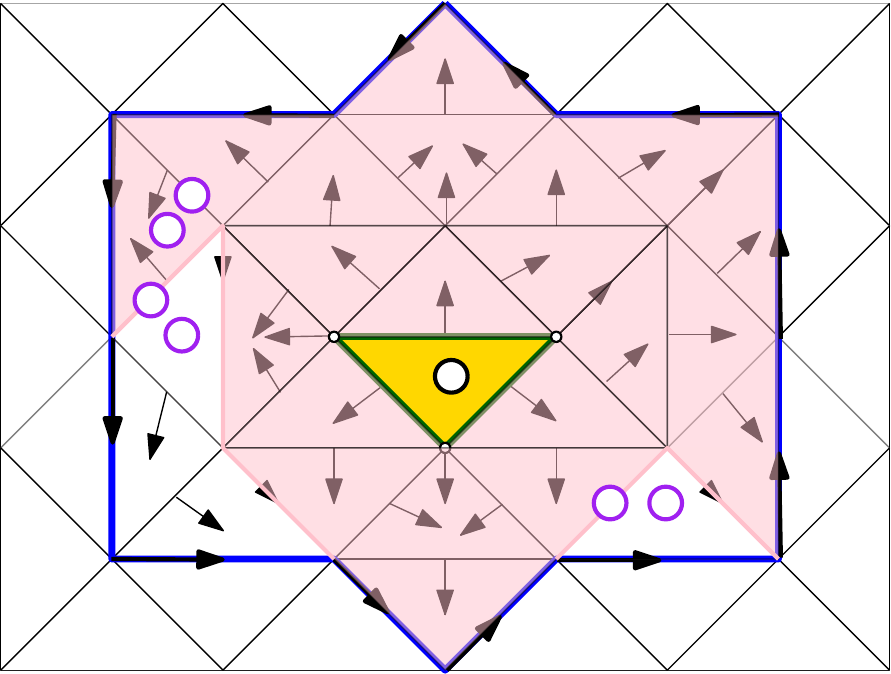}}
\end{tabular}
\caption{Two index pairs for the gold triangle, denoted $\sigma$. The first is given by $(\cl(\sigma), \mo(\sigma))$ where $\mo(\sigma)$ is in green and $\cl(\sigma) \setminus \mo(\sigma)$ is exactly the gold triangle. The second index pair is $(\pf(\cl(\sigma)), \pf(\mo(\sigma)))$, where $\pf(\mo(\sigma))$ consists of those simplices which are colored pink, green, and blue, while $\pf(\cl(\sigma))$ consists of all colored simplices. Note that the second index pair is also an index pair in $N$, where $N$ is taken to be the entire rectangle. }
\label{mouthetc-fig}
\end{figure} 
 \section{Conley Index Persistence}
\label{sec:persist}

We move to establishing the foundations for persistence of the Conley Index. Given a sequence of multivector fields $\mathcal{V}_1,\mathcal{V}_2,\ldots,\mathcal{V}_n$ on a simplicial complex $K$, one may want to quantify the changing behavior of the vector fields. One such approach is to compute a sequence of isolated invariant sets $S_1, S_2, \ldots, S_n$ under each multivector field, and then to compute an index pair for each isolated invariant set. By Proposition \ref{prop:clomo}, a canonical way to do this is to take the closure and mouth of each isolated invariant set to obtain a sequence of index pairs $(\cl(S_1), \mo(S_1)), (\cl(S_2), \mo(S_2)), \ldots, (\cl(S_n), \mo(S_n))$. A first idea is to take the element-wise intersection of consecutive index pairs, which results in the zigzag filtration:
\[
(\cl(S_1), \mo(S_1)) \supseteq (\cl(S_1) \cap \cl(S_2), \mo(S_1) \cap \mo(S_2)) \subseteq (\cl(S_2), \mo(S_2)) \cdots  (\cl(S_n), \mo(S_n)) 
\]
Taking the relative homology groups of the pairs in the zigzag sequence, we obtain a zigzag persistence module. We can extract a barcode corresponding to a decomposition of this module: 
\[\adjustbox{scale = .75}{
    \begin{tikzcd}[column sep = tiny]
        \centering
        H_p(\cl(S_1),\mo(S_1))
        & H_p(\cl(S_1)\cap \cl(S_2), \mo(S_1) \cap \mo(S_2)) \arrow[l] \arrow[r]
        & H_p(\cl(S_2),\mo(S_2))
        & \cdots \arrow[l] \arrow[r]
        & H_p(\cl(S_n),\mo(S_n)).
    \end{tikzcd}
}\]
However, the chance that this approach works in practice is low. In general, two isolated invariant sets $S_1, S_2$ need not overlap, and hence their corresponding index pairs need not intersect. For example, if one were to take the blue periodic solutions in the multivector fields in Figure~\ref{intro-fig} to be $S_1$, $S_2$, $S_3$, by using Proposition \ref{prop:clomo} one gets the index pairs $(S_1,\emptyset)$, $(S_2, \emptyset)$, and $(S_3, \emptyset)$ (since $\cl(S_i) = S_i)$. Note that in such a case, the intermediate pairs are $(S_1 \cap S_2, \emptyset)$ and $(S_2 \cap S_3, \emptyset)$. But $S_1 \cap S_2$ and $S_2 \cap S_3$ intersect only at vertices, so none of the $1$-cycles persist beyond their multivector field. This is problematic in computing the persistence, because intuitively there should be an $H_1$ generator that persists through all three multivector fields. To increase the likelihood that two index pairs intersect, we consider a special type of index pair called an \textit{index pair in $N$}. 
\begin{definition}
    Let $S$ be an invariant set isolated by $N$ under $\mathcal{V}$. The pair of closed sets $(P,E)$ satisfying $E \subseteq P \subseteq N$ is an {\em index pair for $S$ in $N$} if all of the following conditions are met: 
    \begin{enumerate}
        \item $F_{\mathcal{V}}(P) \cap N \subseteq P$ 
        \item $F_{\mathcal{V}}(E) \cap N \subseteq E$
        \item $F_{\mathcal{V}}(P \setminus E) \subseteq N$, and 
        \item $S = \inv(P \setminus E)$. 
    \end{enumerate}
\label{def:modindpair}
\end{definition}
As is expected, such index pairs in $N$ are index pairs. 
\begin{theorem}
Let $(P,E)$ be an index pair in $N$ for $S$. The pair $(P,E)$ is an index pair for $S$ in the sense of Definition \ref{def:indpair}. 
\end{theorem}
\begin{proof}
Note that by condition three of Definition \ref{def:modindpair}, if $\sigma \in P \setminus E$, then $F_{\mathcal{V}}(\sigma) \subseteq N$. Condition one of Definition \ref{def:modindpair} implies that $F_{\mathcal{V}}(\sigma) \cap N = F_{\mathcal{V}}(\sigma) \subseteq P$, which is condition two of Definition \ref{def:indpair}. Likewise, by condition two of Definition \ref{def:modindpair}, if $\sigma \in E$, then $F_{\mathcal{V}}(\sigma) \cap N \subseteq E$. Note that $P \subseteq N$, so it follows that $F_{\mathcal{V}}(\sigma) \cap P \subseteq F_{\mathcal{V}}(\sigma) \cap N \subseteq E$, which is condition one of Definition \ref{def:indpair}. Finally, condition four of Definition \ref{def:modindpair} directly implies condition three of Definition \ref{def:indpair}. 
\end{proof}
An additional advantage to considering index pairs in $N$ is that the intersection of index pairs in $N$ is an index pair in $N$. In general, $(\cl(S_1) \cap \cl(S_2), \mo(S_1) \cap \mo(S_2))$ is not an index pair. However, for index pairs in $N$, we get the next two results which involve the notion of a new multivector field obtained by intersection. Given two multivector fields $\mathcal{V}_1$, $\mathcal{V}_2$, we define $\mathcal{V}_1 \overline{\cap} \mathcal{V}_2 := \{ V_1 \cap V_2 \; | \; V_1 \in \mathcal{V}_1, \; V_2 \in \mathcal{V}_2 \}$.
\begin{theorem}
Let $(P_1,E_1),(P_2,E_2)$ be index pairs in $N$ for $S_1,S_2$ under $\mathcal{V}_1,\mathcal{V}_2$. The set $\inv( (P_1 \cap P_2 ) \setminus (E_1 \cap E_2) )$ is isolated by $N$ under $\mathcal{V}_1 \overline{\cap} \mathcal{V}_2$. 
\end{theorem}
\begin{proof}
To contradict, we assume that there exists a path $\rho \; : \; \mathbb{Z}_{[a,b]} \to N$ under $\mathcal{V}_1\overline\cap\mathcal{V}_2$ where $\rho(a),\rho(b) \in \inv( (P_1 \cap P_2 ) \setminus (E_1 \cap E_2) ) $ and there exists some $i \in (a,b) \cap \mathbb{Z}$ where $\rho(i) \not\in \inv( (P_1 \cap P_2 ) \setminus (E_1 \cap E_2) )$. Note that by the the definition of an index pair, $F_\mathcal{V}(P) \cap N \subseteq P$. Hence, it follows by an easy induction argument that since $F_{\mathcal{V}_1 \overline{\cap} \mathcal{V}_2}(\sigma) \subseteq F_{\mathcal{V}_1}(\sigma), F_{\mathcal{V}_2}(\sigma)$, we have that $\rho(\mathbb{Z}_{[a,b]}) \subseteq P_1, P_2$. This directly implies that $\rho(\mathbb{Z}_{[a,b]}) \subseteq P_1 \cap P_2$. In addition, it is easy to see that $\rho$ can be extended to an essential solution in $P_1 \cap P_2$, which we denote $\rho' \; : \; \mathbb{Z} \to N$, by some simple surgery on essential solutions. This is because there must be essential solutions $\rho_1, \rho_2 \; : \; \mathbb{Z} \to (P_1 \cap P_2) \setminus (E_1 \cap E_2)$ where $\rho_1(a) = \rho(a)$ and $\rho_2(b) = \rho(b)$, as $\rho(a)$ and $\rho(b)$ are both in essential solutions. Hence, $\rho'(x) = \rho_1(x)$ if $x \leq a$, $\rho'(x) = \rho(x)$ if $a \leq x \leq b$, and $\rho'(x) = \rho_2(x)$ if $b \leq x$. Since $\rho'$ is an essential solution, we have that $\rho(\mathbb{Z}_{[a,b]}) \subseteq \inv(P_1 \cap P_2)$, but also that $\rho(\mathbb{Z}_{[a,b]}) \not\subseteq \inv( (P_1 \cap P_2) \setminus (E_1 \cap E_2) )$. Therefore, we must have that $\rho(i) \in E_1 \cap E_2$. But by the same reasoning as before, it follows that $\rho(\mathbb{Z}_{[i,b]}) \subseteq E_1 \cap E_2$. Hence, $b \not\in (P_1 \cap P_2) \setminus (E_1 \cap E_2)$, a contradiction. 
\end{proof}

\begin{theorem}
Let $(P_1, E_1)$ and $(P_2,E_2)$ be index pairs in $N$ under $\mathcal{V}_1,\mathcal{V}_2$. The tuple $(P_1 \cap P_2, E_1 \cap E_2)$ is an index pair for $\inv( (P_1 \cap P_2) \setminus (E_1 \cap E_2) )$ in $N$ under $\mathcal{V}_1 \overline{\cap} \mathcal{V}_2$.
\label{thm:intersect}
\end{theorem}
\begin{proof}
We proceed by using the conditions in Definition \ref{def:modindpair} to show that $(P_1 \cap P_2, E_1 \cap E_2)$ is an index pair in $N$. Note that $F_{\mathcal{V}_1 \overline{\cap} \mathcal{V}_2}(P_1 \cap P_2) \cap N \subseteq F_{\mathcal{V}_1}(P_1) \cap F_{\mathcal{V}_2}(P_2) \cap N$, which is immediate by the definition of $F$ and considering $\mathcal{V}_1 \overline{\cap} \mathcal{V}_2$. Note that since $(P_1,E_1)$ and $(P_2,E_2)$ are index pairs in $N$, we know from Definition \ref{def:modindpair} that $F_{\mathcal{V}_1}(P_1) \cap N \subseteq P_1$ and $F_{\mathcal{V}_2}(P_2) \cap N \subseteq P_2.$ Therefore $F_{\mathcal{V}_1 \overline{\cap} \mathcal{V}_2}(P_1 \cap P_2) \cap N \subseteq P_1 \cap P_2$. This implies the first condition in Definition \ref{def:modindpair}. This argument also implies the second condition by replacing $P$ with $E$. 

Now, we aim to show that $(P_1 \cap P_2, E_1 \cap E_2)$ satisfies condition three in Definition \ref{def:modindpair}. Consider $\sigma \in (P_1 \cap P_2) \setminus (E_ 1 \cap E_2)$. Without loss of generality, we assume $\sigma \not\in E_1$. Therefore, $\sigma \in P_1 \setminus E_1$, so $F_{\mathcal{V}_1}(\sigma) \subseteq N$ by the definition of an index pair in $N$. Hence, since $F_{\mathcal{V}_1 \overline{\cap} \mathcal{V}_2}(\sigma) \subseteq F_{\mathcal{V}_1}(\sigma)$, condition three is satisfied. 

Finally, note that $\inv( (P_1 \cap P_2) \setminus ( E_1 \cap E_2 ) )$ is obviously equal to $\inv( (P_1 \cap P_2) \setminus ( E_1 \cap E_2 ) )$, so condition four holds as well. 
\end{proof}
Hence, if $(P_i, E_i)$ are index pairs in $N$, these theorems gives a meaningful notion of persistence of Conley index through
the decomposition of the following zigzag persistence module:
\begin{eqnarray}\adjustbox{scale = 1}{
    \begin{tikzcd}[column sep = tiny]
        \centering
        H_p(P_1,E_1)
        & H_p(P_1\cap P_2, E_1\cap E_2) \arrow[l] \arrow[r]
        & H_p(P_2,E_2)
        & \cdots \arrow[l] \arrow[r]
        & H_p(P_n,E_n).
    \end{tikzcd}
    \label{zzpers}
}\end{eqnarray}
Because of the previous two theorems, when one decomposes the above zigzag module, one is actually capturing a changing Conley index. This contrasts the case where one only considers index pairs of the form $(\cl(S_i), \mo(S_i))$, because $(\cl(S_i) \cap \cl(S_{i+1}), \mo(S_i) \cap \mo(S_{i+1}))$ need not be an index pair for any invariant set. 

As has been established, the pair $(\cl(S), \mo(S))$ is an index pair, but it need not be an index pair in $N$. We introduce a canonical approach to transform $(\cl(S), \mo(S))$ to an index pair in $N$ by using the \textit{push forward}.
\begin{definition}
    The {\em push forward} $\pf(S)$ of a set $S$ in $N$, $N$ closed, is the set of all simplices in $S$ together with those $\sigma \in N$ such that there exists a path $\rho \; : \; \mathbb{Z}_{[a,b]} \to N$ where $\rho(a) \in S$ and $\rho(b) = \sigma$. 
\end{definition}

If $N$ is not clear from context, we use the notation $\pf_N(S)$. The next series of results imply that an index pair in $N$ can be obtained by taking the push forward of $(\cl(S),\mo(S))$. 
\begin{proposition}
If $S \subseteq K$ is an isolated invariant set with isolating neighborhood $N$ under $\mathcal{V}$, then $\pf(\mo(S)) \cap \cl(S) = \mo(S)$.
\label{lem:pfinter}
\end{proposition}
\begin{proof}
Note that by definition, $\mo(S) \subseteq \cl(S)$ and $\mo(S) \subseteq \pf(\mo(S))$, so it follows that $\mo(S) \subseteq \cl(S) \cap \pf(\mo(S))$. Hence, it is sufficient to show that $\cl(S) \cap \pf(\mo(S)) \subseteq \mo(S)$. Aiming for a contradiction, assume there exists a $\sigma \in \cl(S) \cap \pf(\mo(S))$ where $\sigma \not\in \mo(S)$. This directly implies that $\sigma \in \cl(S) \setminus \mo(S)$.  But by Proposition \ref{prop:clomo}, $\cl(S) \setminus \mo(S) = S$, so $\sigma \in S$. But since $\sigma \in \pf(\mo(S))$, there exists a path $\rho \; : \; \mathbb{Z}_{[a,b]} \to N$ where $\rho(a) \in \mo(S)$ and $\rho(b) = \sigma$. Because $\rho(a) \in \mo(S)$, there exists a $\sigma' \in S$ such that $\rho(a) \leq \sigma'$.  This implies that there exists a path $\rho' \; : \; \mathbb{Z}_{[a-1,b]} \to N$ where $\rho(a-1) = \sigma'$ and $\rho(b) = \sigma$, but $\rho(a) \not\in S$. Hence, $S$ is not isolated by $N$, a contradiction. 
\end{proof}
\begin{proposition}
If $S \subseteq K$ is an isolated invariant set with isolating neighborhood $N$ under $\mathcal{V}$, then $\pf(\mo(S)) \cup \cl(S) = \pf(\cl(S))$.
\label{lem:pfunion}
\end{proposition}
\begin{proof}
Note that since $\pf(\mo(S)) \subseteq \pf(\cl(S))$ and $\cl(S) \subseteq \pf(\cl(S))$, it follows immediately that $\pf(\mo(S)) \cup \cl(S) \subseteq \pf(\cl(S))$. Hence, it is sufficient to show that $\pf(\cl(S)) \subseteq \pf(\mo(S)) \cup \cl(S)$. Aiming for a contradiction, we assume there exists a $\sigma \in \pf(\cl(S))$ such that $\sigma \not \in \pf(\mo(S)) \cup \cl(S)$. Note that since $\sigma \not \in \pf(\mo(S))$, it follows that there must exist a path $\rho \; : \; \mathbb{Z}_{[a,b]} \to N$ where $\rho(i) \not\in\mo(S)$ for all $i$. Else, $\rho(b) = \sigma$ would be in $\pf(\mo(S))$, a contradiction. Hence, there must exist a $\rho(i) \in \cl(S)\setminus\mo(S) = S$ such that $\rho(i+1) \in F_\mathcal{V}(\rho(i))$  and $\rho(i+1) \not\in \cl(S)$. Note that $\rho(i+1)$ cannot be a face of $\rho(i)$, else $\rho(i+1) \in \cl(S)$. Hence, $[\rho(i+1)] = [\rho(i)]$, which means one can construct a path $\rho' \; : \; \mathbb{Z}_{[0,2]} \to N$ where $\rho'(0) = \rho'(2) = \rho(i)$ and $\rho(1) = \rho(i+1)$. But this is a path with endpoints in $S$ which is not contained in $S$, which contradicts $S$ being isolated by $N$. 
\end{proof}
\begin{proposition}
If $S \subseteq K$ is an isolated invariant set with isolating neighborhood $N$, then $\pf(\cl(S)) \setminus \pf(\mo(S)) = \cl(S) \setminus \mo(S) = S$. 
\label{lem:ipeq}
\end{proposition}
\begin{proof}
First, we note that by Proposition \ref{lem:pfunion}, we have that $\pf(\mo(S)) \cup \cl(S) = \pf(\cl(S))$. From Proposition \ref{lem:pfinter} we have that $\pf(\mo(S)) \cap \cl(S) = \mo(S)$, which together with the fact that $\cl(S) \subseteq \mo(S)$ implies that $\pf(\mo(S)) \cup (\cl(S) \setminus \mo(S) ) = \pf(\cl(S))$. Since $\pf(\mo(S))$ and $\cl(S) \setminus \mo(S)$ are disjoint, it follows that that $\cl(S) \setminus \mo(S) = \pf(\cl(S)) \setminus \pf(\mo(S))$. By Proposition \ref{prop:clomo}, $(\cl(S), \mo(S))$ is a saturated index pair for $S$, so it follows that $\pf(\cl(S)) \setminus \pf(\mo(S)) = \cl(S) \setminus \mo(S) = S$. 
\end{proof}
Crucially, from these propositions we get the following. 
\begin{theorem}
If $S$ is an isolated invariant set then $(\pf(\cl(S)), \pf(\mo(S)))$ is an index pair in $N$ for $S$. 
\label{thm:pfip}
\end{theorem}
\begin{proof}
First, we note that since the index pair $(\cl(S),\mo(S))$ is saturated, it follows that $S = \inv(\cl(S) \setminus \mo(S)) = \cl(S) \setminus \mo(S)$. But since by Proposition \ref{lem:ipeq} $\cl(S) \setminus \mo(S) = \pf(\cl(S)) \setminus \pf(\mo(S))$, it follows that $S = \pf(\cl(S)) \setminus \pf(\mo(S)) = \inv( \pf(\cl(S)) \setminus \pf(\mo(S)))$, which satisfies condition four of being an index pair in $N$. 

We show that $F_\mathcal{V}(\pf(\cl(S))) \cap N \subseteq \pf(\cl(S))$. Let $x \in \pf(\cl(S))$, and assume that $y \in F_\mathcal{V}(x) \cap N$. There must be a path $\rho \; : \; \mathbb{Z}_{[a,b]} \to N$ where $\rho(a) \in \cl(S)$ and $\rho(b) = x$, by the definition of the push forward. Thus, we can construct an analogous path $\rho' \; : \; \mathbb{Z}_{[a,b+1]} \to N$ where $\rho'(i) = \rho(i)$ for $i \in \mathbb{Z}_{[a,b]}$ and $\rho'(b+1) = y$. Hence, $y \in \pf(\cl(S))$ by definition. Identical reasoning can be used to show that $F_\mathcal{V}(\pf(\mo(S))) \cap N \subseteq \pf(\mo(S))$, so $(\pf(\cl(S)),\pf(\mo(S)))$ also meets the first two conditions required to be an index pair. 

Finally, we show that $F_\mathcal{V}( \pf(\cl(S)) \setminus \pf(\mo(S)) ) \subseteq N$. By Proposition \ref{lem:ipeq}, this is equivalent to showing that $F_\mathcal{V}( \cl(S) \setminus \mo(S) ) \subseteq N$. Since $(\cl(S),\mo(S))$ is an index pair for $S$, it follows that $F_\mathcal{V}(\cl(S) \setminus \mo(S)) \subseteq \cl(S)$. Note that since $N \supseteq S$ is closed, it follows that $\cl(S) \subseteq N$. Hence, $F_\mathcal{V}( \pf(\cl(S)) \setminus \pf(\mo(S)) ) \subseteq N$, and all conditions for an index pair in $N$ are met. 
\end{proof}

An example of an index pair induced by the push forward can be seen in Figure \ref{mouthetc-fig}. Hence, instead of considering a zigzag filtration given by a sequence of index pairs $(\cl(S_1), \mo(S_1))$, $(\cl(S_2), \mo(S_2))$, \ldots, $(\cl(S_n), \mo(S_n))$, a canonical choice is to instead consider the zigzag filtration given by the the sequence of index pairs \[(\pf(\cl(S_1)), \pf(\mo(S_1))), (\pf(\cl(S_2)), \pf(\mo(S_2))), \ldots, (\pf(\cl(S_n)), \pf(\mo(S_n))).\] 

Choosing $S_i$ is highly application specific, so in our implementation we choose $S_i := \inv_{\mathcal{V}_i}(N)$. This decision together with the previous theorems gives Algorithm \ref{alg:fixedN} for computing the persistence of the Conley Index.
\begin{algorithm2e}
\SetAlgoLined
\KwIn{ Sequence of multivector fields $\mathcal{V}_1, \mathcal{V}_2, \ldots, \mathcal{V}_n$, closed set $N \subseteq K$. } 
\KwOut{ Barcodes corresponding to persistence of Conley Index }
    
    $i \gets 1$
    
    \While{$i <= n$}{
    
        $S_i \gets \inv_{\mathcal{V}_i}(N)$
        
        $\left( P_i, E_i \right) \gets \left( \pf\left( \cl\left( S_i \right) \right), \pf\left( \mo\left(S_i\right) \right) \right)$
    
        $i \gets i + 1$
    }
    
    \Return{ ${\tt zigzagPers}\left( \left(P_1, E_1\right) \supseteq \left(P_1 \cap P_2, E_1 \cap E_2\right) \subseteq \left(P_2, E_2\right) \supseteq \ldots \subseteq \left(P_n, E_n\right) \right)$}

 \caption{Scheme for computing the persistence of the Conley Index, fixed $N$}
 \label{alg:fixedN}
\end{algorithm2e}
Index pairs and barcodes computed by Algorithm \ref{alg:fixedN} can be seen in Figure \ref{fig:constantexp}.

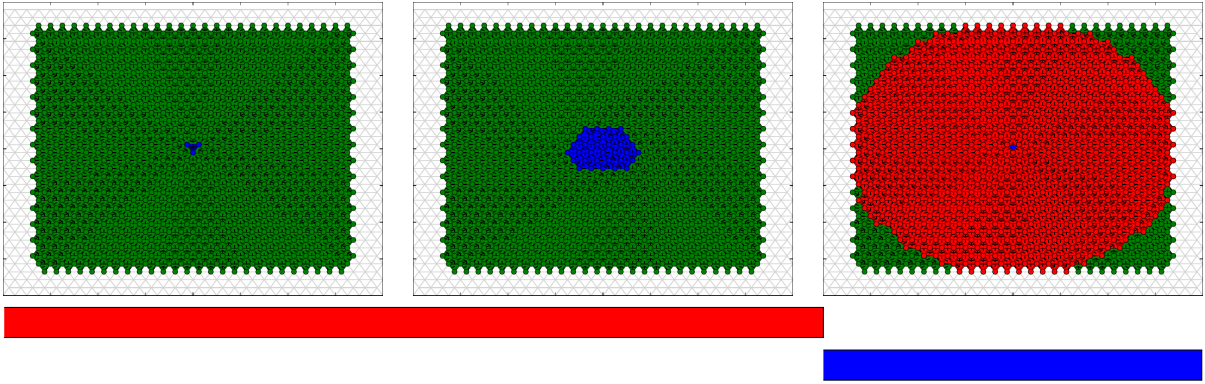
\begin{figure}[htbp]
\centering
\begin{tabular}{ccc}
  \includegraphics[height=39mm]{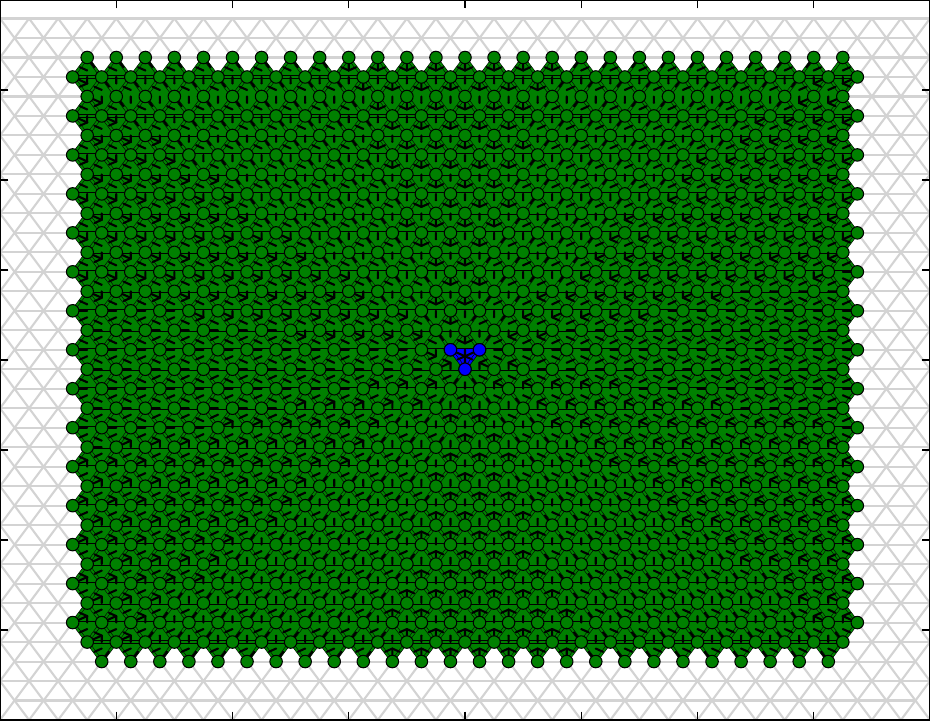}&
  \includegraphics[height=39mm]{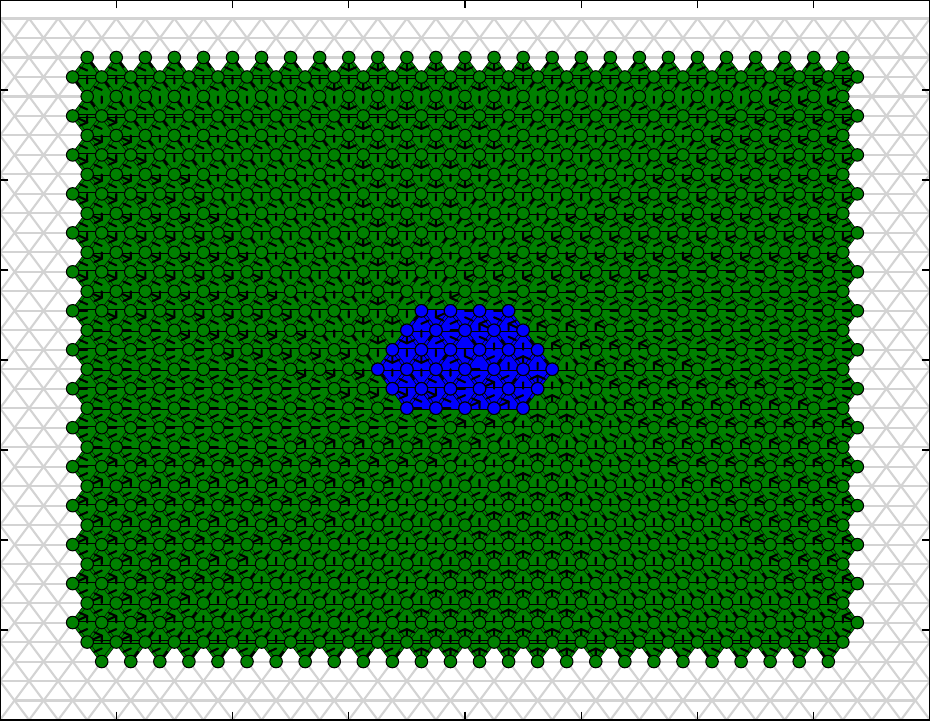}&
  \includegraphics[height=39mm]{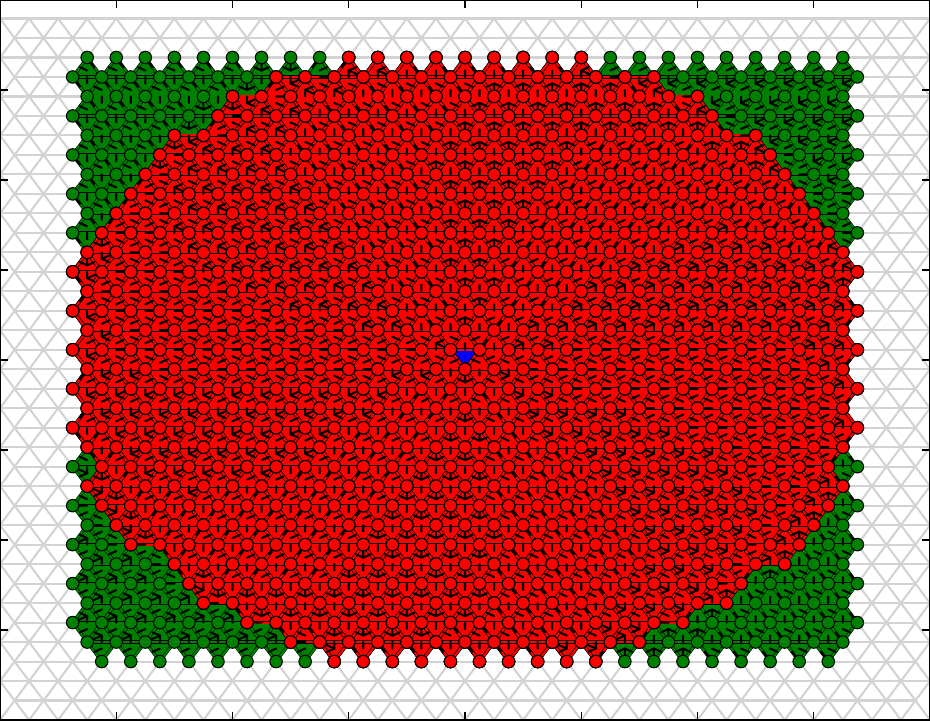}\\
  \multicolumn{3} {l} {
  \begin{tikzpicture}[outer sep = 0, inner sep = 0]
  \draw[fill=red,outer sep=0, inner sep=0] (0.0,-0.3) -- (10.89,-0.3) -- (10.89,0.1) -- (0.0,0.1);
  \end{tikzpicture} } \\
  \multicolumn{3} {r} {
  \begin{tikzpicture}[outer sep = 0, inner sep = 0]
  \draw[fill=blue,outer sep=0, inner sep=0] (-5.03,-0.3) -- (0.0,-0.3) -- (0.0,0.1) -- (-5.03,0.1);
  \end{tikzpicture} } 
\end{tabular}
\caption{Examples of index pairs computed by using the push forward on multivector fields induced by a differential equation. A sequence of multivector fields was generated from a $\lambda$-parametrized differential equation undergoing supercritical Hopf bifurcation \cite[Section 11.2]{HaleKocak1991}. The consecutive images (from left to right) present a selection from this sequence: the case when $\lambda<0$ and there is only an attracting fixed point inside $N$; the case when $\lambda>0$ is small and $N$ contains a repelling fixed point, a small attracting periodic trajectory and all connecting trajectories; the case when $\lambda>0$ is large and the periodic trajectory is no longer contained in $N$. In all three images, we depict $N$ in green, $E$ in red, and $P \setminus E$ in blue. Note that in the leftmost image, the only invariant set is a triangle which represents an attracting fixed point. For this invariant set in this $N$, the only relative homology group which is nontrivial is $H_0(P, E)$, which has a single homology generator. In the middle image, the invariant sets represent a repelling fixed point, a periodic attractor, and heteroclinic orbits which connect the repelling fixed point with the periodic attractor. Note that the relative homology has not changed from the leftmost case, so the only nontrivial homology group is $H_0(P,E)$. In the rightmost image, the periodic attractor is no longer entirely contained within $N$, so the only invariant set corresponds to a repelling fixed point. Here, the only nontrivial homology group is $H_2(P,E)$, which has one generator, so the Conley index has changed. Algorithm \ref{alg:fixedN} captures this change. The persistence barcode output by Algorithm \ref{alg:fixedN} is below index pairs, where a $H_0$ generator (red bar) lasts until the periodic trajectory leaves $N$, at which point it is replaced by an $H_2$ generator (blue bar). }
\label{fig:constantexp}
\end{figure} 

\subsection{Noise-Resilient Index Pairs}

The strategy given for producing index pairs in $N$ produces saturated index pairs. Equivalently, the cardinality of $P \setminus E$ is minimized. This is problematic in the presence of noise, where if $\mathcal{V}_2$ is a slight perturbation of $\mathcal{V}_1$ we frequently have that $\inv_{\mathcal{V}_1}(N) \neq \inv_{\mathcal{V}_2}(N)$. This gives a perturbation in our generated index pairs and in particular a perturbation in $P \setminus E$. As the Conley Index is obtained by taking relative homology, taking the intersection of index pairs $(P_1, E_2)$ and $(P_2, E_2)$ where $P_i \setminus E_i = \inv_{\mathcal{V}_i}(N)$ can result in a ``breaking'' of bars in the barcode. An example can be seen in Figure~\ref{mocl-fig}, where because of noise, the two $P \setminus E$ do not overlap, and hence a $2$-dimensional homology class which intuitively should persist throughout the interval does not. In Figure~\ref{mocl-fig}, the Conley indices of the invariant sets consisting of the singleton critical triangles in $\mathcal{V}_1$ and $\mathcal{V}_2$ (the left and right multivector fields) have rank $1$ in dimension $2$ because the homology group $H_2$ of $P$ (which is the entire complex in both cases) relative to $E$ (which is all pink simplices) has rank $1$. However, the generators for $H_2(P_1,E_1)$ and $H_2(P_2,E_2)$ are both in the intersection field $\mathcal{V}_1\overline{\cap}\mathcal{V}_2$. Hence, rather than one generator persisting through all three multivector fields, we get two bars that overlap at the intersection field. The difficulty is rooted in the fact that the sets $W_1=P_1\setminus E_1$, $W_2=P_2\setminus E_2$, and $W_{12}=(P_1\cap P_2)\setminus (E_1\cap E_2)$ do not have a common intersection.

\begin{figure}[htbp]
\centering
\begin{tabular}{ccc}
  \includegraphics[height=34mm]{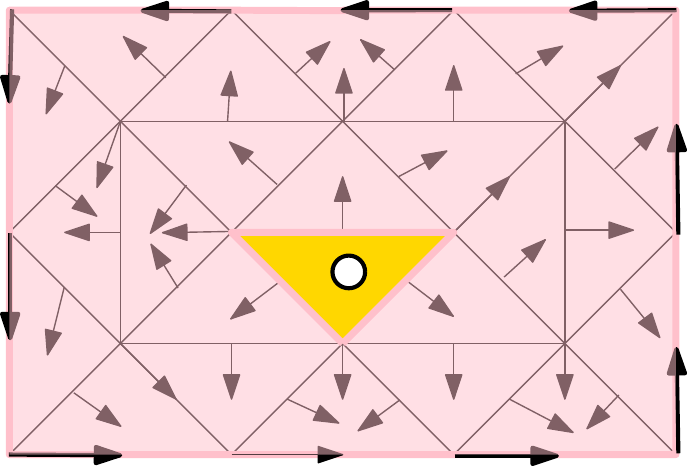}&
  \includegraphics[height=34mm]{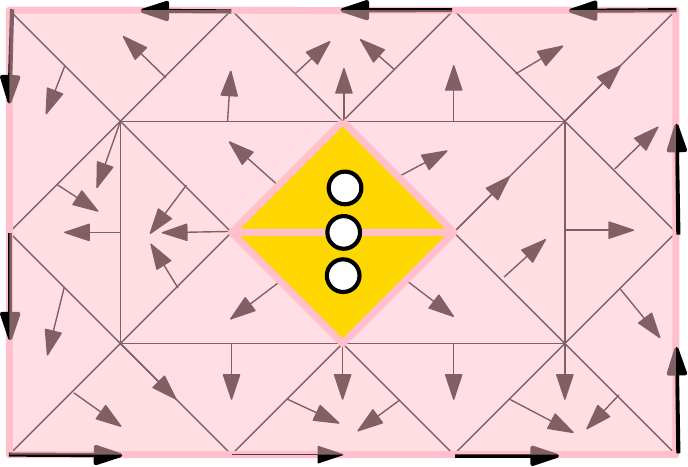}&
  \includegraphics[height=34mm]{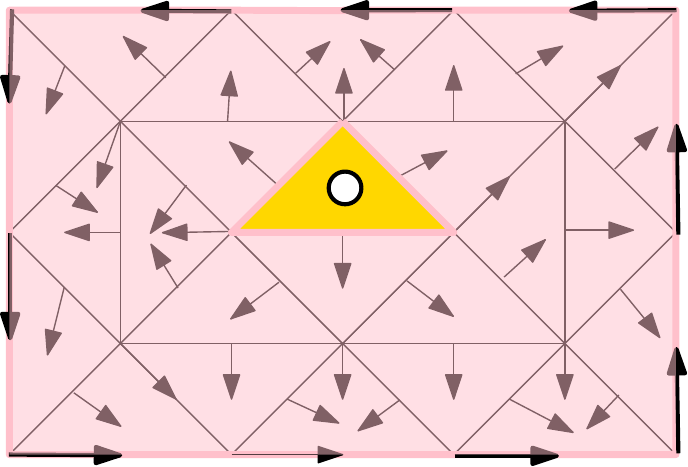}\\
  \multicolumn{3} {l} {
  \begin{tikzpicture}[outer sep = 0, inner sep = 0]
  \draw[fill=blue,outer sep=0, inner sep=0] (0.0,-0.3) -- (10.9,-0.3) -- (10.9,0.1) -- (0.0,0.1);
  \end{tikzpicture} } \\
  \multicolumn{3} {r} {
  \begin{tikzpicture}[outer sep = 0, inner sep = 0]
  \draw[fill=blue,outer sep=0, inner sep=0] (-10.9,-0.3) -- (0.0,-0.3) -- (0.0,0.1) -- (-10.9,0.1);
  \end{tikzpicture} }
\end{tabular}
\caption{Infeasibilty of the index pair $(\pf(\cl(S)),\pf(\mo(S)))$: The sets $E=\pf(\mo(S))$ are colored pink in all three images, while the invariant sets which equal $P\setminus E$ are golden in all three images. (left) ${\mathcal{V}_1}:  P_{1} \setminus E_{1}$ consists of a single golden triangle;
(right) ${\mathcal{V}_2}$:  $P_2\setminus E_2$ consists of the single golden triangle; (middle) $(P_1\cap P_2)\setminus (E_1\cap E_2)$ consists of two golden triangles (excluding the edge between them) in the intersection field $\mathcal{V}_1\overline{\cap}\mathcal{V}_2$. The barcode for index pairs is depicted by two blue bars, each of which represents a $2$-dimensional homology generator. Ideally, these would be a single bar. }
\label{mocl-fig}
\end{figure}
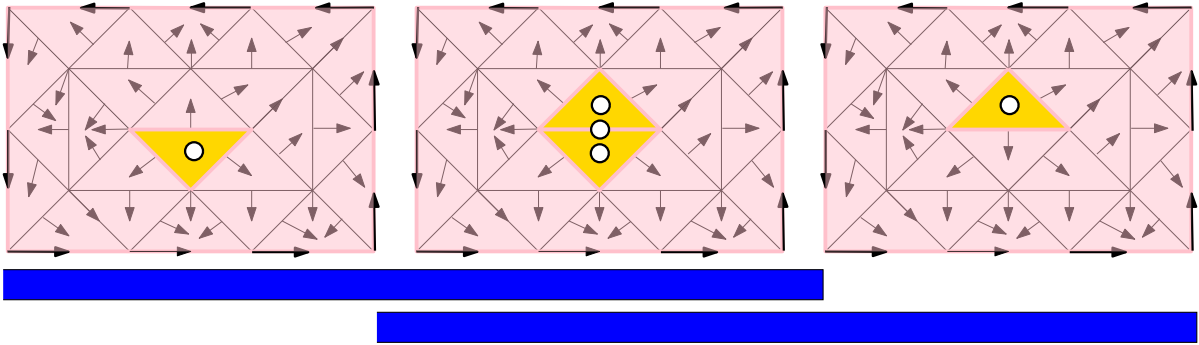 

To address this problem, we propose an algorithm to expand the size of $P \setminus E$. It is important to note that a balance is needed to ensure a large $E$ as well as a large $P \setminus E$. If $E_1$ and $E_2$ are too small, then it is easy to see that $E_1$ and $E_2$ may not intersect as expected even though consecutive vector fields are very similar. 
The following proposition is very useful for computing a balanced index pair.
\begin{proposition}
Let $(P,E)$ be an index pair for $S$ in $N$ under $\mathcal{V}$. If $V \subseteq E$ is a regular multivector where $E':= E\setminus V$ is closed, then $(P,E')$ is an index pair for $S$ in $N$. 
\label{enlarge-prop}
\end{proposition}
\begin{proof}
Note that since $P$ does not change, it is immediate the $P$ satisfies $F_\mathcal{V}(P) \cap N \subseteq P$, and the first condition of an index pair in $N$ is met. 

We show that $F_\mathcal{V}(E') \cap N \subseteq E$. For a contradiction, we assume that there exists an $x \in E'$ such that there is a $y \in F_\mathcal{V}(x) \cap N$, $y \not\in E'$. Note that if $y \leq x$, then by definition $y$ is in the closure of $x$. Since $E$ is closed and $x \in E$, it follows that $y \in E$. But since $y \not \in E'$, it follows that $y \in V$. But this is a contradiction since by assumption, $E \setminus V$ is closed. Hence, by definition of $F_\mathcal{V}$, since $y \not\in \cl(x)$, $y$ and $x$ must be in the same multivector. Note that in such a case $y \not \in V$, as if it were, $x$ would not be in $E'$ since $x$ and $y$ are in the same multivector. This implies that $F_\mathcal{V}(E) \cap N \not\subset E$, as $y \in F_\mathcal{V}(x)$ and $y \not\in E$, a contradiction. Thus, we conclude that $F_\mathcal{V}(E') \cap N \subseteq E$.

Now, we show that $F_\mathcal{V}(P \setminus E') \subseteq N$. Assume there exists an $x \in P$ satisfying $F_\mathcal{V}(x) \setminus N \neq \emptyset$. Then since $(P,E)$ is an index pair in $N$, it follows that $x \in E$. To contradict, we assume $x \not\in E'$. Hence, $x \in V$. We let $y \in F_\mathcal{V}(x) \setminus N$. By definition of $F_\mathcal{V}$, either $y \in \cl(x)$ or $y$ and $x$ are in the same multivector. In the later case, $y \in N$ as it was assumed that $V \subseteq E$, a contradiction. Hence, $y \leq x$. But this implies that $y \in \cl(x) \subseteq E \subseteq N$, a contradiction. Ergo, $x \in E'$, and $F_\mathcal{V}(P \setminus E') \subseteq N$. 

Finally, we show that $\inv(P \setminus E') = \inv(P \setminus E)$. Trivially, $\inv(P \setminus E) \subseteq \inv(P \setminus E')$, so it is sufficient to show that $\inv(P \setminus E') \subseteq \inv(P \setminus E)$. For a contradiction, assume that there exists an $x \in \inv( P \setminus E' ), x \not\in \inv( P \setminus E)$. Thus, there exists an essential solution $\rho \; : \; \mathbb{Z} \to P \setminus E'$ where for some $k$, $y := \rho(k) \in V$. Since $V$ is regular, we assume without loss of generality that $z := \rho(k+1) \not\in V$. Hence, $z \in \cl(y)$. In addition, since $z \in \cl(V)$, we have that $z \in \cl(E) = E$. Therefore, $z \in E \setminus V = E'$. But $\rho$ is a solution in $P \setminus E'$, a contradiction. 
\end{proof}

\begin{figure}[htbp]
\centering
\begin{tabular}{ccc}
  \includegraphics[height=34mm]{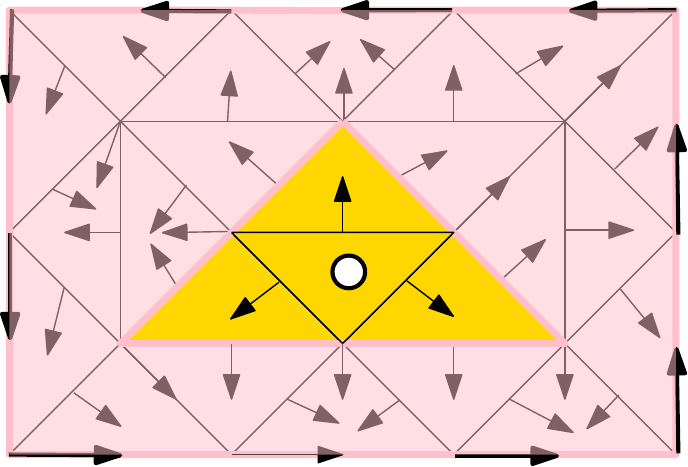}&
  \includegraphics[height=34mm]{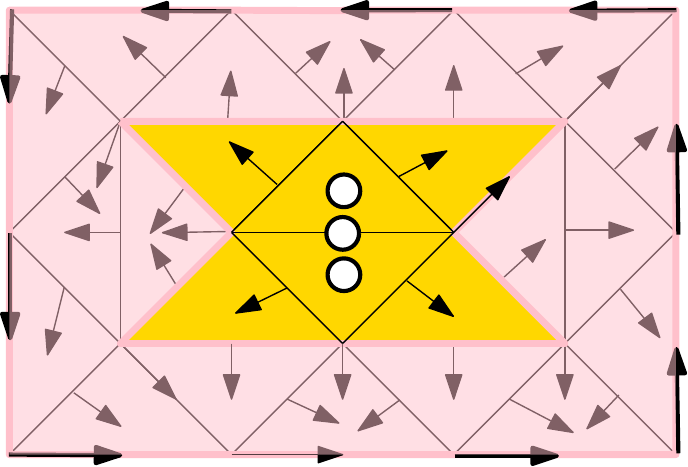}&
  \includegraphics[height=34mm]{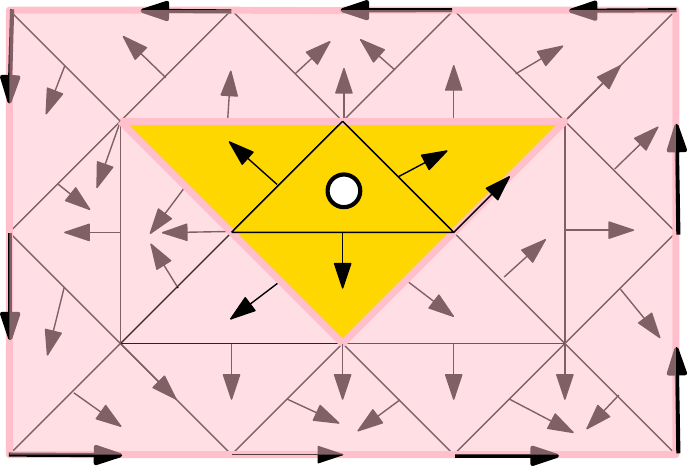}\\
  \multicolumn{3} {l} {
  \begin{tikzpicture}[outer sep = 0, inner sep = 0]
  \draw[fill=blue,outer sep=0, inner sep=0] (0.0,-0.3) -- (15.9,-0.3) -- (15.9,0.1) -- (0.0,0.1);
  \end{tikzpicture} } 
\end{tabular}
\caption{Enlarging $P\setminus E$ which is gold in all three pictures while $E$ is colored pink. (left) $\mathcal{V}_1$; (right) $\mathcal{V}_2$; (middle) $\mathcal{V}_1\overline{\cap}\mathcal{V}_2$. Note that there is one bar in the barcode, in contrast with Figure \ref{mocl-fig}}
\label{thickPE-fig}
\end{figure}
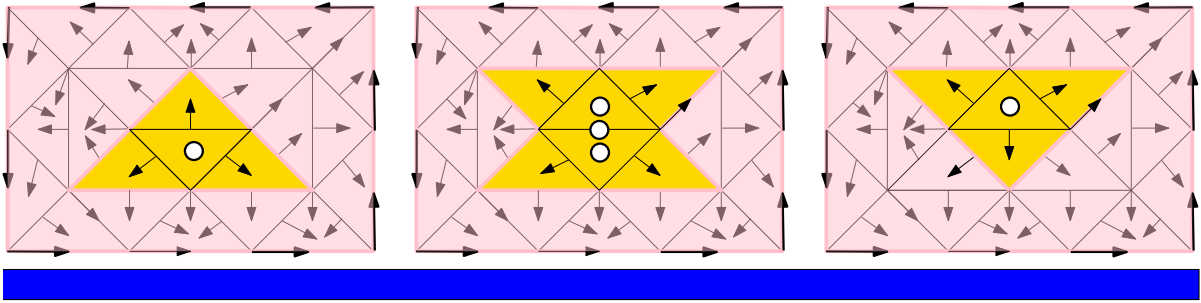 
Figure~\ref{thickPE-fig} illustrates how enlarging $P\setminus E$ by removing regular vectors as Proposition~\ref{enlarge-prop} suggests can help mitigate the effects of noise on computing Conley index persistence. Contrast this example with the example in Figure~\ref{mocl-fig}. Denoting  
$W_i=P_i\setminus E_i$ for $i=1,2$ and $W_{12}=(P_1\cap P_2)\setminus (E_1\cap E_2)$ in both figures, we see that
$W_1\cap W_{12}\cap W_2$ is empty in Figure~\ref{mocl-fig} while in Figure~\ref{thickPE-fig} it consists of three critical simplices each marked
with a circle. Hence, in Figure \ref{thickPE-fig} a single generator persists throughout the interval, unlike in Figure \ref{mocl-fig}.

\subsection{Computing a Noise-Resilient Index Pair}

We give a method for computing a noise-resilient index pair by using techniques demonstrated in the previous subsection. Note that by Theorem \ref{thm:pfip}, we have that $( \pf(\cl(S)), \pf(\mo(S)))$ is an index pair for invariant set $S$ in $N$. Hence, we adopt the strategy of taking $P = \pf(\cl(S))$ and $E = \pf(\mo(S))$, and we aim to find some collection $R \subseteq E$ so that $(P, E \setminus R)$ remains an index pair in $N$. Finding an appropriate $R$ is a difficult balancing act: one wants to find an $R$ so that $P \setminus ( E \setminus R )$ is sufficiently large, so as to capture perturbations in the isolated invariant set as described in the previous section, but not so large that $E$ is small and perturbations in $E$ are not captured. If $R$ is chosen to be as large as possible, then a small shift in $E$ may results in $(E \setminus R) \cap (E' \setminus R')$ having a different topology than $E$ or $E'$ leading to a ``breaking'' of barcodes analogous to the case described in the previous section. 

Before we give an algorithm for outputting such an $R$, we first define a $\delta$-collar. 
\begin{definition}
We define the {\em $\delta$-collar} of an invariant set $S \subseteq K$ recursively: 
\begin{enumerate}
    \item The $0$-collar of $S$ is $\cl(S)$.
    \item For $\delta > 0$, the $\delta$-collar of $S$ is the set of simplices $\sigma$ in the $\left(\delta - 1\right)$-collar of $S$ together with those simplices $\tau$ where $\tau$ is a face of some $\sigma$ with a face $\tau'$ in the $(\delta - 1)$-collar of $S$. 
\end{enumerate}
\end{definition}
For an isolated invariant set $S$, we will let $C_\delta(S)$ denote the $\delta$-collar of $S$. Together with Proposition \ref{enlarge-prop}, $\delta$-collars give a natural algorithm for finding an $R$ to enlarge $P \setminus E$. 

\begin{algorithm2e}
\SetAlgoLined
\KwIn{ Isolated invariant set $S$ with respect to $\mathcal{V}$ contained in some closed set $N$, Index pair $(P,E)$ in $N$ with respect to $\mathcal{V}$, $\delta \in \mathbb{Z}$} 
\KwOut{ List of simplices $R$ such that $(P, E \setminus R)$ is an index pair for $S$ in $N$. }
    
        $R \gets  {\tt new \; set()}$
        
        $vecSet \gets \{ \left[\sigma\right] \in \mathcal{V} \; | \; \left[\sigma\right] \subseteq E \cap C_\delta(S) \wedge \left[\sigma\right] \cap \partial(E) \neq \emptyset \wedge \left[\sigma\right] \cap \partial(P) = \emptyset \}$
    
        $vec \gets {\tt new \; queue}()$
    
        ${\tt appendAll}(vec, vecSet)$
        
        \While{  ${\tt size}( vec ) > 0 $  }
        {
            $\left[\sigma\right] \gets {\tt pop}(vec)$
        
            \If{${\tt isClosed}( \left(E \setminus R\right) \setminus \left[\sigma\right] ) \; {\tt and } \; \left[\sigma\right] \subseteq E \setminus R \; {\tt and } \; {\tt isRegular}\left( \left[\sigma\right] \right)$ }
            {
                $R \gets R \cup \left[\sigma\right]$
                
                $mouthVecs \gets \{ \left[\tau\right] \; | \; \tau \in \mo\left(\left[\sigma\right]\right) \; \wedge \left[\tau\right] \subseteq C_\delta(S) \} $
                
                ${\tt appendAll}(vec,mouthVecs)$
            }
        }

    \Return{R}

 \caption{${\tt findR}(S, P, E, \mathcal{V}, \delta)$}
 \label{alg:findR}
\end{algorithm2e}
In particular, we use Algorithm \ref{alg:findR} for this purpose. 
\begin{theorem}
Let $R$ be the output of Algorithm \ref{alg:findR} applied to index pair $(P,E)$ in $N$ for isolated invariant set $S$. The pair $(P, E \setminus R)$ is an index pair for $S$ in $N$.
\label{thm:algoworks}
\end{theorem}
\begin{proof}
To contradict, we assume that the $R$ output by Algorithm \ref{alg:findR} results in $(P, E \setminus R)$ is not an index pair. We note by inspection of the algorithm that multivectors are removed sequentially, so there exists some first $V$ such that $(P, E \setminus R_V)$ is an index pair for $S$ in $N$ but $(P, E \setminus \left(R_V \cup V\right))$ is not an index pair for $S$ in $N$, where $R_V$ denotes the $R$ variable in Algorithm \ref{alg:findR} before $V$ is added to it. By inspection of the algorithm, we observe that since $V$ was added to $R_V$, it must be that $V$ is a regular vector, and $E \setminus \left(R \cup V\right)$ is closed, and $V \subseteq E \setminus R_V$. Proposition \ref{enlarge-prop} directly implies that $\left(P, \left(E\setminus R\right) \setminus V \right)$ is an index pair, which is a contradicton. Hence, there can exist no such $V$, so it follows that $(P, E \setminus R)$ must be an index pair for $S$ in $N$. 
\end{proof}

Hence, Algorithm \ref{alg:findR} provides a means by which the user may enlarge $P \setminus E$. As this algorithm is parameterized, a robust choice of $\delta$ may be application specific. We also include some demonstrations on the effectiveness of using this technique. A real instance of the difficulty can be seen in Figure \ref{fig:noise}, while the application of Algorithm \ref{alg:findR} with $\delta = 5$ to solve the problem is found in Figure \ref{fig:solvednoise}.
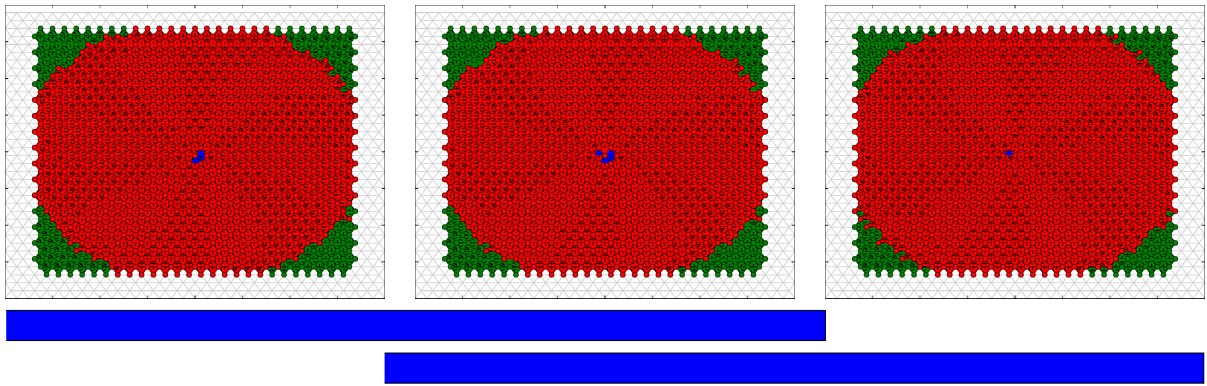
\begin{figure}[htbp]
\centering
\begin{tabular}{ccc}
  \includegraphics[height=39mm]{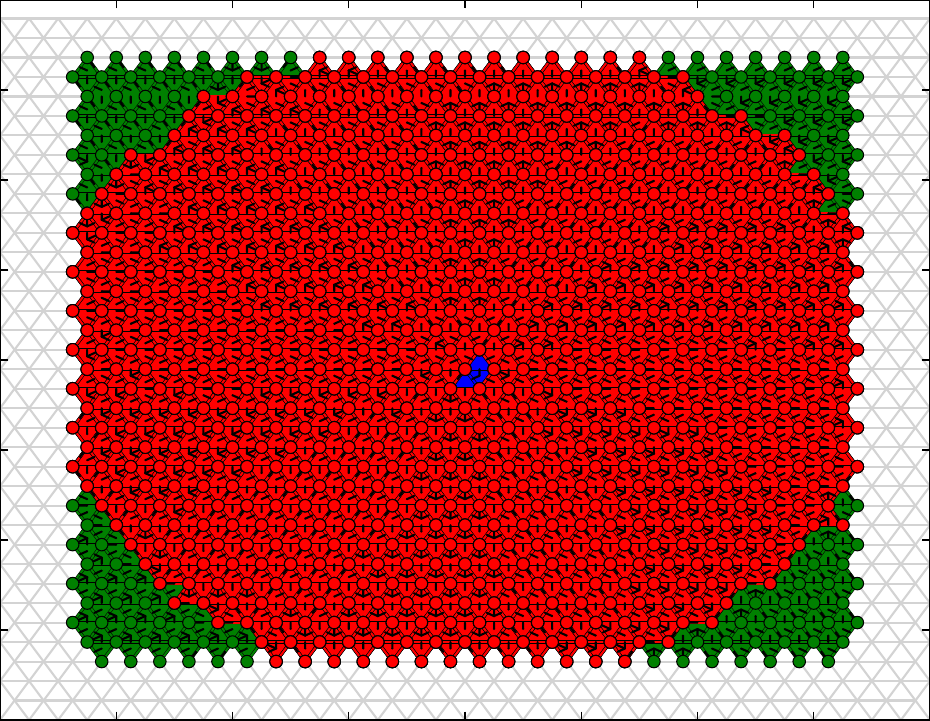}&
  \includegraphics[height=39mm]{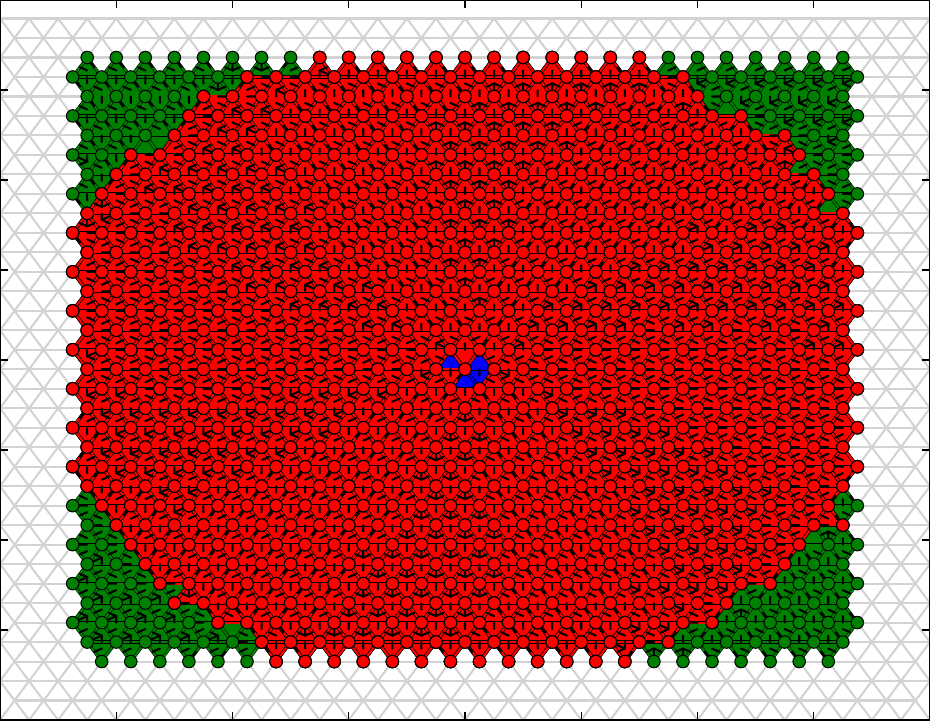}&
  \includegraphics[height=39mm]{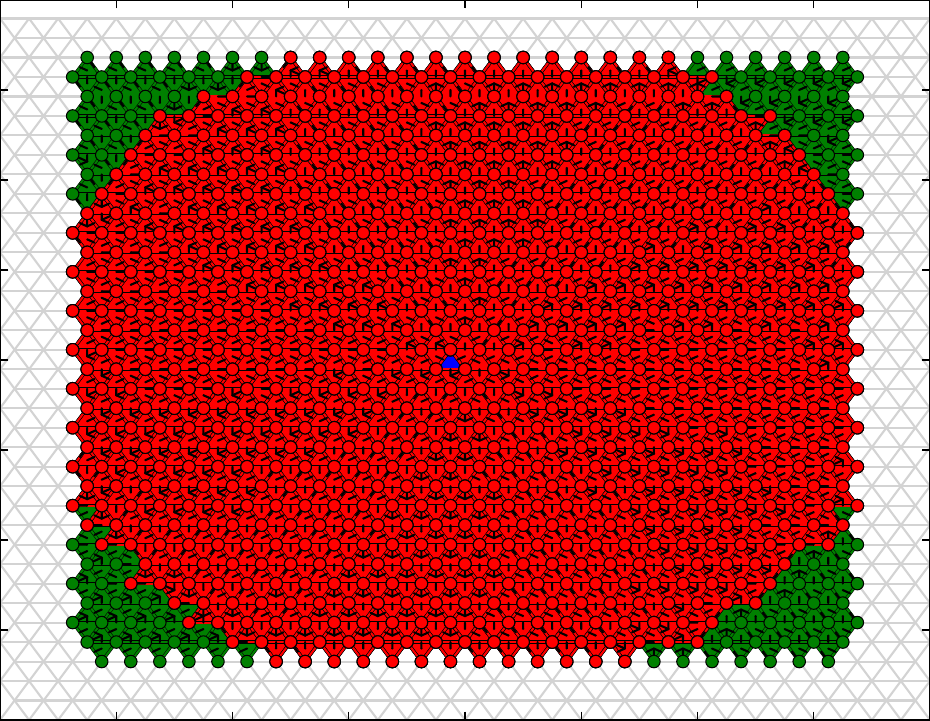}\\
  \multicolumn{3} {l} {
  \begin{tikzpicture}[outer sep = 0, inner sep = 0]
  \draw[fill=blue,outer sep=0, inner sep=0] (0.0,-0.3) -- (10.89,-0.3) -- (10.89,0.1) -- (0.0,0.1);
  \end{tikzpicture} } \\
  \multicolumn{3} {r} {
  \begin{tikzpicture}[outer sep = 0, inner sep = 0]
  \draw[fill=blue,outer sep=0, inner sep=0] (-10.89,-0.3) -- (0.0,-0.3) -- (0.0,0.1) -- (-10.89,0.1);
  \end{tikzpicture} }
\end{tabular}
\caption{Index pairs on two slightly perturbed multivector fields (left, right) and their intersection (middle). As before, the isolating neighborhood $N$ is in green, $E$ is in red, and $P \setminus E$ is in blue. Note that we have the same difficulty as in Figure \ref{mocl-fig}, where there are two homology generators in the intersection multivector field, so we get a broken bar code. }
\label{fig:noise}
\end{figure} 

\begin{figure}[htbp]
\centering
\begin{tabular}{ccc}
  \includegraphics[height=39mm]{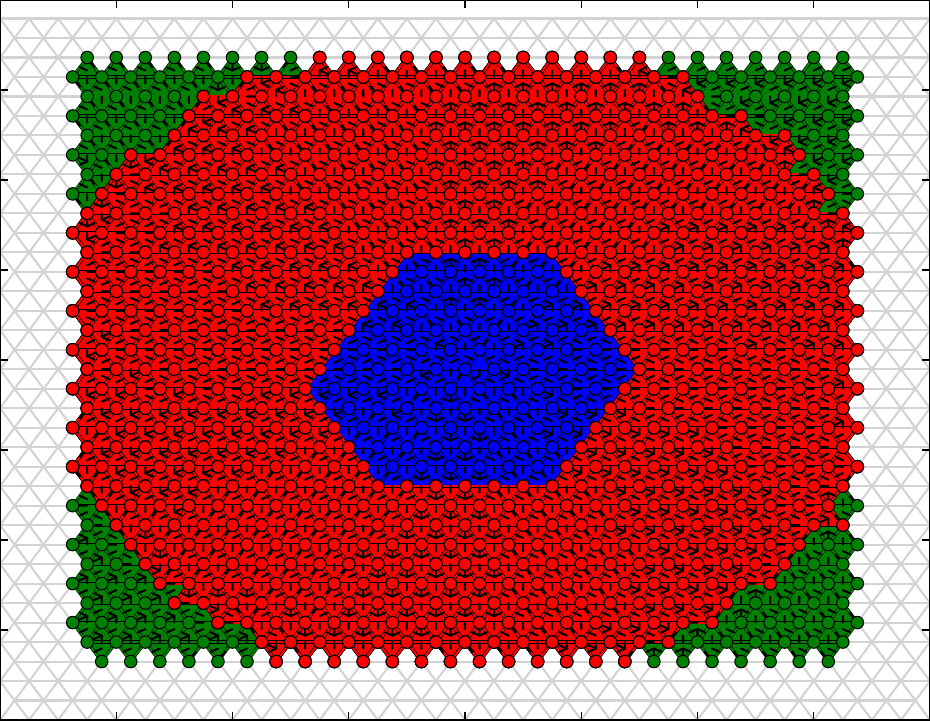}&
  \includegraphics[height=39mm]{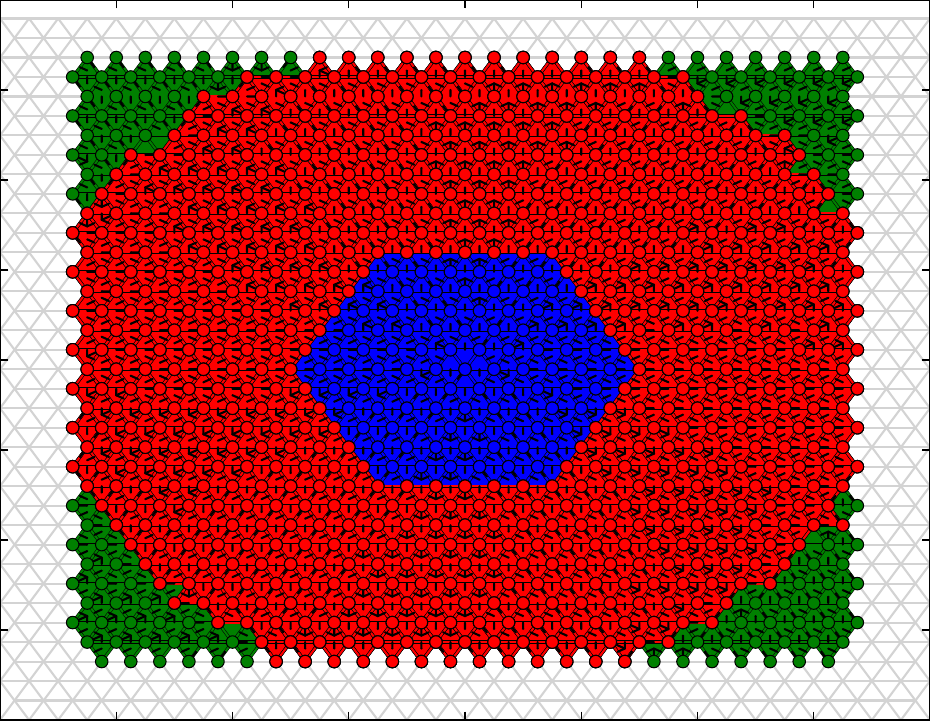}&
  \includegraphics[height=39mm]{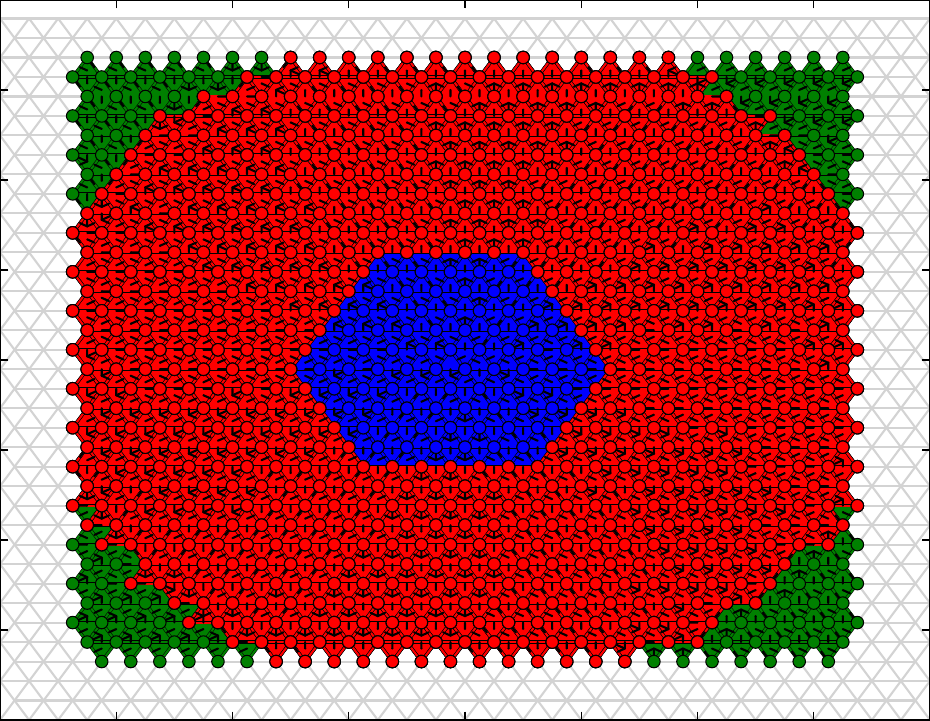}\\
  \multicolumn{3} {l} {
  \begin{tikzpicture}[outer sep = 0, inner sep = 0]
  \draw[fill=blue,outer sep=0, inner sep=0] (0.0,-0.3) -- (15.95,-0.3) -- (15.95,0.1) -- (0.0,0.1);
  \end{tikzpicture} } 
\end{tabular}
\caption{The same index pairs as in Figure \ref{fig:noise} with the same color scheme, but after applying Algorithm \ref{alg:findR} to reduce the size of $E$. This forces a $2$-dimensional homology generator to persist across both multivector fields (left, right) and their intersection (middle). }
\label{fig:solvednoise}
\end{figure}
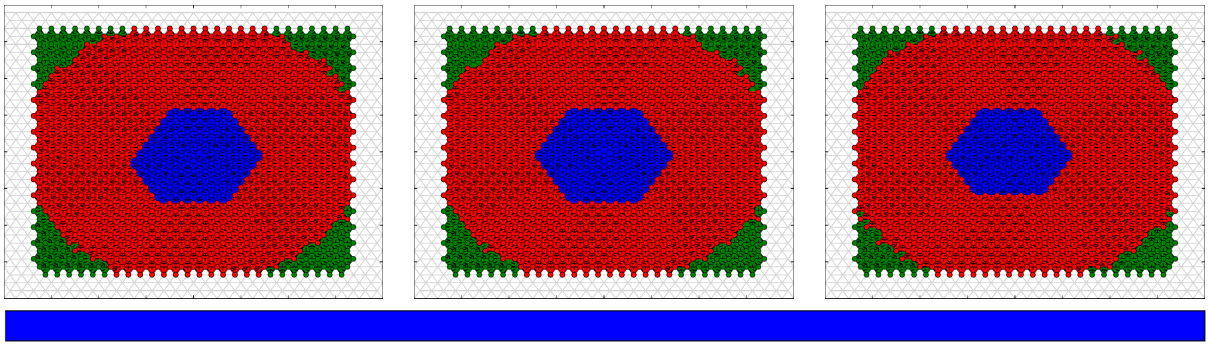 

\section{Tracking Invariant Sets}
\label{set:track}
In the previous section, we established the persistence of the Conley index of invariant sets in consecutive multivector fields which are isolated by a single isolating neighborhood. In this section, we develop an algorithm to ``track'' an invariant set over a sequence of
isolating neighborhoods. A classic example is a hurricane, where if one were to sample wind velocity at times $t_0, t_1, \ldots, t_n$, there may be no fixed $N$ which captures the eye of the hurricane at all $t_i$ without also capturing additional, undesired invariant sets at some $t_j$. 

\subsection{Changing the Isolating Neighborhood}

Thus far, we have defined a notion of persistence of the Conley Index for some fixed isolating neighborhood $N$ and simplicial complex $K$. This setting is very inflexible ---one may want to incorporate domain knowledge to change $N$ so as to capture changing features of a sequence of sampled dynamics. We now extend the results in Section \ref{sec:persist} to a setting where $N$ need not be fixed. Throughout this section, we consider multivector fields $\mathcal{V}_1, \mathcal{V}_2, \ldots, \mathcal{V}_n$ with corresponding isolated invariant sets $S_1, S_2, \ldots, S_n$. In addition, we assume that there exist isolating neighborhoods $N_1, N_2, \ldots, N_{n-1}$ where $N_i$ isolates both $S_i$ and $S_{i+1}$. We will also require that if $1 < i < n$, the invariant set $S_i$ is isolated by $N_i \cup N_{i+1}$.
Note that for each $i$ where $1 < i < n$, there exist two index pairs for $S_i$: one index pair $(P_i^{(i-1)},E_i^{(i-1)})$ in $N_{i-1}$ and another index pair $(P_i^{(i)},E_i^{(i)})$ in $N_i$. In the case of $i=1$, there is only one index pair $(P_1^{(1)},E_1^{(1)})$ for $S_i$. Likewise, in the case of $i = n$, there is a single index pair $(P_n^{(n-1)},E_n^{(n-1)})$. 

By applying the techniques of Section \ref{sec:persist}, we obtain a sequence of persistence modules:
\begin{equation}\adjustbox{scale = .8}{
    \begin{tikzcd}
        \centering
        H_p\left(P_1^{(1)},E_1^{(1)}\right)
        & H_p\left(P_1^{(1)} \cap P_2^{(1)}, E_1^{(1)} \cap E_2^{(1)}\right) \arrow[l] \arrow[r]
        & H_p\left(P_2^{(1)},E_2^{(1)}\right)\\
        H_p\left(P_2^{(2)},E_2^{(2)}\right)
        & H_p\left(P_2^{(2)} \cap P_3^{(2)}, E_2^{(2)} \cap E_3^{(2)}\right) \arrow[l] \arrow[r]
        & H_p\left(P_3^{(2)},E_3^{(2)}\right)\\
        H_p\left(P_3^{(3)},E_3^{(3)}\right)
        & H_p\left(P_3^{(3)} \cap P_4^{(3)}, E_3^{(3)} \cap E_4^{(3)}\right) \arrow[l] \arrow[r]
        & H_p\left(P_4^{(3)},E_4^{(3)}\right)\\
        & \vdots & \\
        H_p\left(P_{n-1}^{(n-1)},E_{n-1}^{(n-1)}\right)
        & H_p\left(P_{n-1}^{(n-1)} \cap P_n^{(n-1)}, E_{n-1}^{(n-1)} \cap E_n^{(n-1)}\right) \arrow[l] \arrow[r]
        & H_p\left(P_n^{(n-1)},E_n^{(n-1)}\right).
    \end{tikzcd}
    \label{eq:bigMod}
}\end{equation}
In the remainder of this subsection, we develop the theory necessary to combine these modules into a single module. Without any loss of generality, we will combine the first modules into a single module, which will imply a method to combine all of the modules into one. 

First, we note that by Theorem \ref{thm:ipiso}, we have that $H_p(P_2^{(1)},E_2^{(1)}) \cong H_p(P_2^{(2)},E_2^{(2)})$. To combine the persistence modules
\begin{equation}
    \begin{tikzcd}
        \centering
        H_p(P_1^{(1)},E_1^{(1)})
        & H_p(P_1^{(1)} \cap P_2^{(1)}, E_1^{(1)} \cap E_2^{(1)}) \arrow[l] \arrow[r]
        & H_p(P_2^{(1)},E_2^{(1)})\\
        H_p(P_2^{(2)},E_2^{(2)})
        & H_p(P_2^{(2)} \cap P_3^{(2)}, E_2^{(2)} \cap E_3^{(2)}) \arrow[l] \arrow[r]
        & H_p(P_3^{(2)},E_3^{(2)}).
    \end{tikzcd}
    \label{eq:smallMod}
\end{equation}
into a single module, it is either necessary to explicitly find a simplicial map which induces an isomorphism $\phi \; : \; H_p(P_2^{(1)},E_2^{(1)}) \to H_p(P_2^{(2)},E_2^{(2)})$, or to construct some other index pair for $S_2$ denoted $(P,E)$ such that $P_2^{(1)},P_2^{(2)} \subset P$ and $E_2^{(1)},E_2^{(2)} \subset E$. This would allow substituting both occurrences of $(P_2^{(1)},E_2^{(1)})$ or $(P_2^{(2)},E_2^{(2)})$ for $(P,E)$, and allow the combining of all the modules in Equation \ref{eq:bigMod} into a single module. Since constructing the isomorphism given by Theorem \ref{thm:ipiso} is fairly complicated, we opt for the second approach. First, we define a special type of index pair that is sufficient for our approach. 

\begin{definition}[Strong Index Pair]
Let $(P,E)$ be an index pair for $S$ under $\mathcal{V}$. The index pair $(P,E)$ is a {\em strong index pair} for $S$ if for each $\tau \in E$, there exists a $\sigma \in S$ such that there is a path $\rho \; : \; \mathbb{Z}_{[a,b]} \to P$ where $\rho(a) = \sigma$ and $\rho(b) = \tau$. 
\end{definition}
Intuitively, a strong index pair $(P,E)$ for $S$ is an index pair for $S$ where each simplex $\tau \in E$ is reachable from a path originating in $S$. Strong index pairs have the following useful property.
\begin{theorem}
Let $S$ denote an invariant set isolated by $N$, $N'$, and $N \cup N'$ under $\mathcal{V}$. If $(P,E)$ and $(P',E')$ are strong index pairs for $S$ in $N$, $N'$ under $\mathcal{V}$, then the pair \[(\pf_{N \cup N'}\left(P \cup P'\right), \pf_{N \cup N'}\left(E \cup E'\right) )\] is a strong index pair for $S$ in $N \cup N'$ under $\mathcal{V}$. 
\label{thm:pfsip}
\end{theorem}
\begin{proof}
We proceed by showing that the pair meets the requirements to be a strong index pair in $N \cup N'$. First, note that for all $\sigma \in \pf(P \cup P')$, if $\tau \in F_\mathcal{V}(\sigma) \cap \left(N \cup N'\right)$, then it follows by the definition of the push forward that $\tau \in \pf(P \cup P')$. Hence, $F_\mathcal{V}\left( \pf\left(P \cup P'\right) \right) \cap \left( N \cup N' \right) \subset \pf\left(P \cup P'\right)$. Analogous reasoning shows that $F_\mathcal{V}\left( \pf\left( E \cup E' \right) \right) \cap \left(N \cup N'\right) \subset \pf\left( E \cup E' \right)$. 

Hence, we proceed to show that $F_\mathcal{V}\left( \pf\left( P \cup P' \right) \setminus \pf\left( E \cup E' \right) \right) \subset N \cup N'$.
To contradict, assume that there exists a $\sigma \in \pf\left( P \cup P' \right) \setminus \pf\left( E \cup E' \right)$ so that there is a $\tau \in F_\mathcal{V}\left(\sigma\right)$ where $\tau \not\in N \cup N'$. Since $\sigma \in \pf\left(P \cup P'\right)$, there must exist a $x \in P \cup P'$ such that there is a path $\rho \; : \; \mathbb{Z}_{[a,b]} \to N \cup N'$ where $\rho(a) = x$ and $\rho(b) = \sigma$. Without loss of generality, we assume that $x \in P$. Note that if $\sigma \in P$,  
then $\sigma\in P\setminus E$ by our assumption and hence $F_\mathcal{V}(\sigma)\subseteq N\subseteq N\cup N'$ contradicting our assumption that $F_\mathcal{V}(\sigma)\ni\tau\not\in (N\cup N')$.
 Hence, $\sigma \not\in P$. Note that since $(P,E)$ is an index pair, we have that $F_\mathcal{V}\left(P \setminus E \right) \subset P$. Hence, there must be some $i$ such that $\rho(i) \in E$. Hence, $\sigma \in \pf(E \cup E')$, a contradiction. Thus, no such $\sigma$ can exist. 

Now, we show that $S = \inv\left( \pf\left( P \cup P' \right) \setminus \pf \left( E \cup E' \right) \right)$. Note that since $S$ is isolated by $N \cup N'$ and every $\sigma \in E \cup E'$ is reachable by a path originating at $\tau \in S$, it follows that $\pf\left(E \cup E'\right) \cap S = \emptyset$, else $S$ is not isolated by $N \cup N'$. Hence, this implies that $S \subset \inv\left( \pf\left( P \cup P' \right) \setminus \pf\left( E \cup E'\right) \right)$. 

To show that $\inv\left( \pf\left( P \cup P' \right) \setminus \pf\left( E \cup E'\right) \right) \subset S$, we first note that if $x \in \pf\left( P \cup P' \right)$ and $x \not \in P \cup P'$, then it follows that $x \in \pf(E \cup E')$. This is because if there exists a $\sigma \in P$ such that there is a path $\rho \; : \; \mathbb{Z}_{[a,b]} \to N \cup N'$ which satisfies $\rho(a) = \sigma$ and $\rho(b) = x$, there must be some $\rho(i) \in E$. If there is no such $\rho(i)$, then this contradicts requirement $2$ of Definition \ref{def:indpair}, which states that $F_\mathcal{V}\left(P \setminus E\right) \subset P$. Hence, by the definition of the push forward, it follows that $x \in \pf\left(E \cup E'\right)$. Thus, it follows that $\pf\left( P \cup P' \right) \setminus \pf\left( E \cup E' \right) \subset \left( P \cup P' \right) \setminus \left( E \cup E' \right)$. Thus, every essential solution in $\pf\left( P \cup P' \right) \setminus \pf\left( E \cup E' \right)$ is also an essential solution in $\left(P \cup P\right) \setminus \left( E \cup E' \right)$. It remains to be shown that every essential solution in $\left( P \cup P' \right) \setminus \left( E \cup E' \right)$ is also an essential solution in $P \setminus E$. 

For a contradiction, assume that there exists an essential solution $\rho \; : \; \mathbb{Z} \to \left( P \cup P' \right) \setminus \left( E \cup E' \right)$ such that there exists an $i$ where $\rho(i) \not \in P \setminus E$. Note that since $\rho(i) \in \left(P \cup P'\right) \setminus \left( E \cup E'\right)$, it follows that $\rho(i) \not\in E$. Together, these two facts imply that $\rho(i) \not\in P$. Hence, it follows that $\rho(i) \in P' \setminus E'$. We claim in particular that for all $j \in \mathbb{Z}$, $\rho(j) \in P' \setminus E'$. Note that if there exists a $j < i$ such that $\rho(j) \not\in P' \setminus E'$, then it follows that $\rho(j) \in P \setminus E$. Since $(P,E)$ is an index pair, we have that $F_\mathcal{V}\left(P \setminus E\right) \subseteq P$. Thus, there must exist some $k$ where $j < k < i$ where $\rho(k) \in E$, but this contradicts that $\rho\left( \mathbb{Z} \right) \subseteq \left(P \cup P'\right) \setminus \left( E \cup E' \right)$. Hence, no such $j$ exists. Analogous reasoning shows that there is no $j > i$ where $\rho(j) \not\in P' \setminus E'$. It follows that $\rho(\mathbb{Z}) \subseteq \inv\left(P' \setminus E'\right)$. But this implies that $\rho\left( \mathbb{Z} \right) \subseteq S$. Note that $(P,E)$ and $(P',E')$ are both index pairs for $S$, so it follows that $\rho\left( \mathbb{Z} \right) \subseteq P \setminus E, P'\setminus E'$. This contradicts our assumption that $\rho(i) \not\in P \setminus E$. Thus, there is no such $\rho$. Hence, we have that $\inv\left(\left( P \cup P' \right) \setminus \left( E \cup E' \right)\right) \subseteq \inv\left( P \setminus E\right), \inv\left( P' \setminus E' \right) = S$, which implies that $\inv\left( \pf\left( P \cup P' \right) \setminus \pf\left( E \cup E' \right) \right) \subseteq S$.

To see that $ \left( \pf\left( P \cup P' \right), \pf\left(E \cup E'\right) \right)$ is a strong index pair, note that there must exist a path from $S$ to every $\sigma \in E \cup E'$, and since $\pf\left( E \cup E' \right)$ is the set of simplices $\sigma$ for which there exists a path from $E \cup E'$ to $\sigma$, it follows easily from path surgery that there exists a path from $S$ to $\sigma$. Hence, $(\pf\left(P \cup P'\right), \pf\left(E \cup E'\right) )$ is a strong index pair for $S$ in $N$. 
\end{proof}

Crucially, this theorem gives a persistence module
\begin{equation}
    \begin{tikzcd}
        \centering
        H_p(P_2^{(1)},E_2^{(1)}) \arrow[r] & H_p\left( \pf\left( P_2^{(1)} \cup P_2^{(2)} \right), \pf\left( E_2^{(1)} \cup E_2^{(2)}  \right) \right) & H_p(P_2^{(2)},E_2^{(2)}) \arrow[l]
    \end{tikzcd}
    \label{eq:connectingMod}
\end{equation}
where the arrows are given by the inclusion. Note that since these are all index pairs for the same $S$, it follows that we have $H_p\left(P_2^{(1)},E_2^{(1)}\right) \cong H_p\left( \pf\left( P_2^{(1)} \cup P_2^{(2)} \right), \pf\left( E_2^{(1)} \cup E_2^{(2)}  \right) \right) \cong H_p\left(P_2^{(2)},E_2^{(2)}\right)$. Hence, we will substitute $H_p\left( \pf\left( P_2^{(1)} \cup P_2^{(2)} \right), \pf\left( E_2^{(1)} \cup E_2^{(2)}  \right) \right)$ into the persistence module. By using the modules in Equation \ref{eq:bigMod}, we get a new sequnece of persistence modules
\begin{equation}\adjustbox{scale = .6}{
    \begin{tikzcd}
        \centering
        H_p\left(P_1^{(1)},E_1^{(1)}\right)
        & H_p\left(P_1^{(1)} \cap P_2^{(1)}, E_1^{(1)} \cap E_2^{(1)}\right) \arrow[l] \arrow[r]
        & H_p\left( \pf\left( P_2^{(1)} \cup P_2^{(2)} \right),\pf\left( E_2^{(1)} \cup E_2^{(2)} \right) \right)\\
        H_p\left( \pf\left( P_2^{(1)} \cup P_2^{(2)} \right),\pf\left( E_2^{(1)} \cup E_2^{(2)} \right) \right)
        & H_p\left(P_2^{(2)} \cap P_3^{(2)}, E_2^{(2)} \cap E_3^{(2)}\right) \arrow[l] \arrow[r]
        & H_p\left( \pf\left( P_3^{(2)} \cup P_3^{(3)} \right),\pf\left( E_3^{(2)} \cup E_3^{(3)} \right) \right)\\
        H_p\left( \pf\left( P_3^{(2)} \cup P_3^{(3)} \right),\pf\left( E_3^{(2)} \cup E_3^{(3)} \right) \right)
        & H_p\left(P_3^{(3)} \cap P_4^{(3)}, E_3^{(3)} \cap E_4^{(3)}\right) \arrow[l] \arrow[r]
        & H_p\left( \pf\left( P_4^{(3)} \cup P_4^{(4)} \right),\pf\left( E_4^{(3)} \cup E_4^{(4)} \right) \right)\\
        & \vdots & \\
        H_p\left( \pf\left( P_{n-1}^{(n-2)} \cup P_{n-1}^{(n-1)} \right),\pf\left( E_{n-1}^{(n-2)} \cup E_{n-1}^{(n-1)} \right) \right)
        & H_p\left(P_{n-1}^{(n-1)} \cap P_n^{(n-1)}, E_{n-1}^{(n-1)} \cap E_n^{(n-1)}\right) \arrow[l] \arrow[r]
        & H_p\left(P_n^{(n-1)},E_n^{(n-1)}\right).
    \end{tikzcd}
    \label{eq:subs}
}\end{equation}
which can immediately be combined into a single persistence module.

This approach is not without it's disadvantages, however. Namely, if $(P,E)$ and $(P',E')$ are index pairs for $S$ in $N$ and $N'$, it requires that $(P,E)$ and $(P',E')$ are strong index pairs and that $S$ is isolated by $N \cup N'$. Fortunately, the push forward approach to computing an index pair in $N$ gives a strong index pair.
\begin{theorem}
Let $S$ be an isolated invariant set where $N$ is an isolating neighborhood for $S$. The pair $\left(\pf\left( \cl\left(S\right) \right), \pf\left( \mo\left( S \right) \right) \right)$ is a strong index pair in $N$ for $S$. 
\end{theorem}
\begin{proof}
We note that by Theorem \ref{thm:pfsip}, the pair $\left(\pf\left( \cl\left(S\right) \right), \pf\left( \mo\left( S \right) \right) \right)$ is an index pair for $S$ in $N$. Hence, it is sufficient to show that the index pair is strong. Note that by definition, for all $\sigma \in \mo(S)$, there exists a $\tau \in S$ such that $\sigma$ is a face of $\tau$. Hence, $\sigma \in F_\mathcal{V}(\tau)$. Note that $\pf(\mo(S))$ is precisely the set of simplices $\sigma'$ for which there exists a path originating in $\mo(S)$ and terminating at $\sigma'$, so it is immediate that there is a path originating in $S$ and terminating at $\sigma$. Hence, the pair $\left(\pf\left( \cl\left(S\right) \right), \pf\left( \mo\left( S \right) \right) \right)$ is a strong index pair. 
\end{proof}
Our enlarging scheme given in Algorithm \ref{alg:findR} does not affect the strongness of an index pair. 
\begin{theorem}
Let $R$ be the output of applying Algorithm \ref{alg:findR} to the strong index pair $\left(P,E\right)$ in $N$ for $S$ with some parameter $\delta$. The pair $\left(P, E\setminus R\right)$ is a strong index pair for $S$ in $N$. 
\end{theorem}
\begin{proof}
Theorem \ref{thm:algoworks} gives that $\left(P, E\setminus R\right)$ is an index pair for $S$ in $N$, so it is sufficient to show that such an index pair is strong. Note that $P$ does not change, but the strongness of index pairs only requires paths to be in $P$. Since all paths in $\left(P,E\right)$ are also paths in $\left(P, E\setminus R\right)$, it follows that $\left(P, E\setminus R\right)$ is a strong index pair in $N$.
\end{proof}

These theorems give us a canonical scheme for choosing invariant sets from a sequence of multivector fields and then computing the barcode of 
persistence module given in Equation~(\ref{eq:subs}). We give our exact scheme in Algorithm \ref{alg:scheme}. 

\begin{algorithm2e}
\SetAlgoLined
\KwIn{ Sequence of multivector fields $\mathcal{V}_1, \mathcal{V}_2, \ldots, \mathcal{V}_n$, closed set $N_0 \subset K$, $\delta \in \mathbb{Z}$. } 
\KwOut{ Barcodes corresponding to persistence of Conley Index }
    
    $i \gets 1$
    
    \While{$i <= n$}{
    
        $S_i \gets \inv_{\mathcal{V}_i}(N_{i-1})$
    
        $\left(P'_{i,1}, E'_{i,1}\right) \gets \left( \pf_{N_{i-1}}\left(\cl\left(S_i\right)\right), \pf_{N_{i-1}}\left(\mo\left(S_i\right)\right) \right)$
        
        $R_{i,1} \gets {\tt findR}(S_i, P'_{i,1}, E'_{i,1}, \mathcal{V}, \delta)$
        
        $\left(P^{(1)}_i, E^{(1)}_i\right) \gets \left(P'_{i,1}, E'_{i,1} \setminus R_{i,1}\right)$
        
        $N_i \gets {\tt find}(S_i, N_{i-1}, \mathcal{V}, \delta)$
        
        $\left(P'_{i,2}, E'_{i,2}\right) \gets \left( \pf_{N_{i}}\left(\cl\left(S_i\right)\right), \pf_{N_{i}}\left(\mo\left(S_i\right)\right) \right)$
        
        $R_{i,2} \gets {\tt findR}(S_i, P'_{i,2}, E'_{i,2}, \mathcal{V}, \delta)$
        
        $\left(P^{(2)}_i, E^{(2)}_i\right) \gets \left(P'_{i,2}, E'_{i,2} \setminus R_{i,2}\right)$
        
        \uIf{$i = 1$}{
            $\left( P_i, E_i \right) \gets \left( P^{(2)}_i, E^{(2)}_i \right)$
        }\uElseIf{$i = n$}{
            $\left( P_i, E_i \right) \gets \left( P^{(1)}_i, E^{(1)}_i \right)$
        }\Else{
            $\left( P_i, E_i \right) \gets \left( \pf_{N_{i-1} \cup N_i}\left(P^{(1)}_i \cup P^{(2)}_i\right), \pf_{N_{i-1} \cup N_i}\left( E^{(1)}_i \cup E^{(2)}_i \right) \right)$
        }

        $i \gets i + 1$
    }
    
    \Return{ ${\tt zigzagPers}\left( \left(P_1, E_1\right) \supseteq \left(P_1^{(2)} \cap P_2^{(1)}, E_1^{(2)} \cap E_2^{(1)}\right) \subseteq \left(P_2, E_2\right) \supseteq \ldots \subseteq \left(P_n, E_n\right) \right)$}

 \caption{Scheme for computing the persistence of the Conley Index, variable $N$}
 \label{alg:scheme}
\end{algorithm2e}

The astute reader will notice an important detail about Algorithm \ref{alg:scheme}. Namely, the $\texttt{find}$ function is parameterized by a nonnegative integer $\delta$, and the function has not yet been defined. In particular, said function must output a closed $N_i \supseteq S_i$ such that $S_i$ is isolated by $N_{i-1} \cup N_i$. An obvious choice is to let $N_i := N_{i-1}$, but such an approach does not allow one to capture essential solutions that ``move'' outisde of $N_{i-1} = N_i$ as the multivector fields change. We give a nontrivial \texttt{find} function in the next subsection that can be used to capture such changes in an essential solution. 
\subsection{Finding Isolating Neighborhoods}

Given an invariant set $S$ isolated by $N$ with respect to $\mathcal{V}$, we now propose a method to find a closed, nontrivial $N' \subseteq K$ such that $N \cup N'$ isolates $S$. Our method relies heavily on the concept of $\delta$-collar introduced in Section \ref{sec:persist}. In fact, we will let $N' = C_\delta(S) \setminus R$ such that $N \cup N'$ isolates $S$. Hence, it is sufficient to devise an algorithm to find $C_\delta(S) \setminus R$. Before we give and prove the correctness of the algorithm, we briefly introduce the notion of the \textit{push backward} of some set $S$ in $N$, denoted $\pb_N(S)$. We let $\pb_N(S) = \{ x \in N \; | \; \exists \; \rho \; : \; \mathbb{Z}_{[a,b]} \to N, \; \rho(a) = x, \rho(b) \in S \}$. Essentially, the push backward of $S$ in $N$ is the set of simplices $\sigma \in N$ for which there exists a path in $N$ from $\sigma$ to $S$.   

\begin{algorithm2e}
\SetAlgoLined
\KwIn{ Invariant set $S$ isolated by $N$ under $\mathcal{V}$, $\delta \in \mathbb{Z}$ } 
\KwOut{ Closed set $N' \supseteq S$ such that $N \cup N'$ isolates $S$ under $\mathcal{V}$}
    $V \gets {\tt new \; stack()}$
    
    $R \gets {\tt new \; set()}$
    
    $pb \gets \pb_N(S)$
    
    \ForEach{$\sigma \in C_\delta(S) \cup N$}{
        ${\tt setUnvisited}(\sigma)$
    }
    
    \ForEach{ $\sigma \in S$ }
    {
        $adj \gets \cl(\sigma) \cup \left[v\right]_\mathcal{V}$
    
        \ForEach{$\tau \in adj$}
        {
            \If{ $\tau \not\in S  \; {\tt and} \; \tau \in C_\delta(S) \cup N $}
            {
                ${\tt push}(V, \tau)$
            }
        }
    }
    
    \While{ ${\tt size}(V) > 0$ }{
        $v \gets {\tt pop}(V)$
        
        \If{$!{\tt hasBeenVisited}(v)$}
        {
            ${\tt setVisited}(v)$
            
            \uIf{ $\left(\cl\left(v\right) \cup \left[ v \right]_\mathcal{V}\right) \cap pb \neq \emptyset$ }
            {
                ${\tt add}(R, v)$
                
                $cf \gets {\tt cofaces}(v)$

                ${\tt addAll}(R, cf)$
                
            }
            \Else{
                \ForEach{$\sigma \in \left(\cl\left(v\right) \cup \left[ \sigma \right]_\mathcal{V}\right) \cap \left(C_\delta(S) \cup N \right)$}{
                    ${\tt push}(V, \sigma)$
                }
            }
        }
    }
    
    \Return{ $C_\delta(S) \setminus R$ }

 \caption{${\tt find}(S, N, \mathcal{V}, \delta)$}
 \label{alg:findisolating}
\end{algorithm2e}

We now prove that $N \cup \left(C_\delta(S)\right) \setminus R$ isolates $S$. Note that since $S \subseteq C_\delta(S) \setminus R$, this also implies that $ C_\delta(S) \setminus R$ isolates $S$.
\begin{theorem}
Let $S$ denote an invariant set isolated by $N \subseteq K$ under $\mathcal{V}$. If $C_\delta(S) \setminus R$ is the output of Algorithm \ref{alg:findisolating} on inputs $S, N, \mathcal{V}, \delta$, then the closed set $N \cup \left(C_\delta(S) \setminus R\right)$ isolates $S$.  
\end{theorem}
\begin{proof}
For a contradiction, assume that there exists a path $\rho \; : \; \mathbb{Z}_{[ a, b]} \to N \cup \left(C_\delta(S) \setminus R\right)$ so that $\rho(a), \rho(b) \in S$ where there is an $i$ satisfying $a < i < b$ with $\rho(i) \not\in S$. Note that since $N$ isolates $S$, if $N \cup C_\delta(S) \setminus R$ does not isolate $S$, then there must exist a first $k \in \mathbb{Z}_{[a,b]}$ such that $\rho(k) \in C_\delta(S)\setminus N$ and  $F_\mathcal{V}(\rho(k)) \cap \pb_N(S) \neq \emptyset$. If this were not the case, then $N$ would not isolate $S$.  Without loss of generality, we assume that for all $a < j < k$, we have that $\rho(j) \not\in S$. Note that for all $j \in \mathbb{Z}_{[a+1,k-1]}$, when $\rho(j)$ is removed from the stack $V$, if $\rho(j+1)$ has not been visited, then $\rho(j+1)$ is added to the stack. Hence, this implies that if any $\rho(j)$ is visited, then $\rho(k)$ will be added to $R$. If this were not the case, there would exist some $\rho(j)$ such that when $\rho(j)$ was removed from the stack, $\rho(j+1)$ was not visited and was not added to the stack. This implies that $F_\mathcal{V}(\rho(j)) \cap \pb_N(S) \neq \emptyset$, which contradicts $\rho(k)$ being the first such simplex in the path. 

Hence, since $\rho(a+1)$ is added to the stack, it follows that $\rho(k)$ is added to $R$, which implies that $\rho\left( \mathbb{Z}_{[a,b]} \right) \not\subset N \cup \left( C_\delta(S) \setminus R \right)$. Note too that $N \cup C_\delta(S) \setminus R$ must be closed, as if there is a $\sigma \in N$ such that $\rho(k) \leq \sigma$, then $\rho(k) \in N$ because $N$ is closed, a contradiction. But when $\rho(k)$ is removed from $C_\delta(S)$, any of its cofaces which are in $C_\delta(S)$ are also removed. Hence, $N$ is closed, $C_\delta(S) \setminus R$ is closed, so their union must be closed.  
\end{proof}

Hence, we use Algorithm \ref{alg:findisolating} as the \texttt{find} function in our scheme given in Algorithm \ref{alg:scheme}. We give an example of our implementation of Algorithm \ref{alg:scheme} using the ${\tt find}$ function in Figure \ref{fig:collar}. 

\begin{figure}[htbp]
\centering
\begin{tabular}{ccc}
  \includegraphics[height=39.4mm]{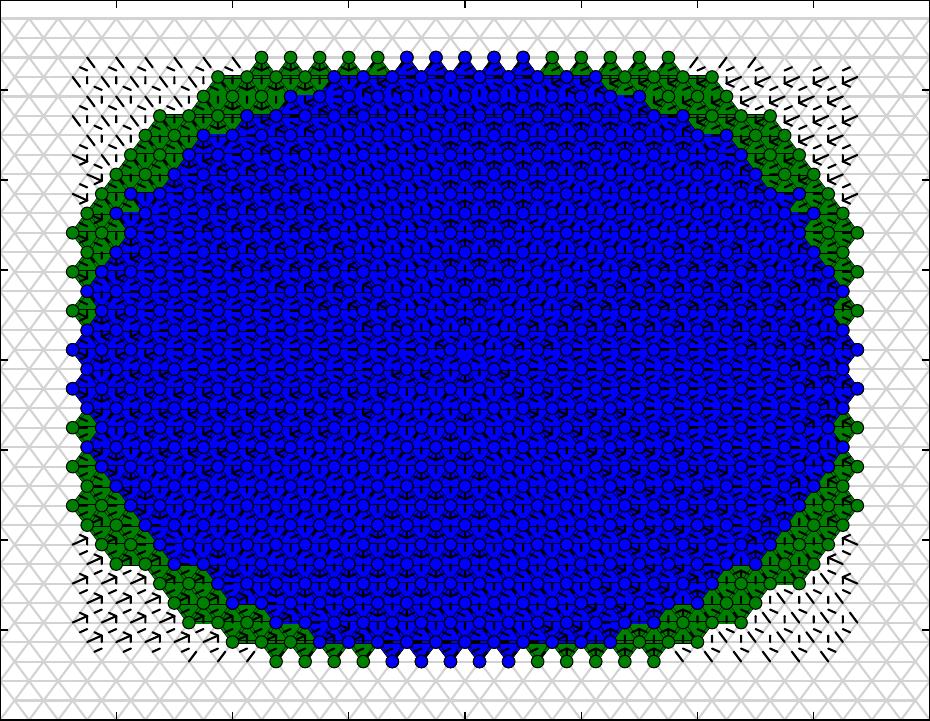}&
  \includegraphics[height=39.4mm]{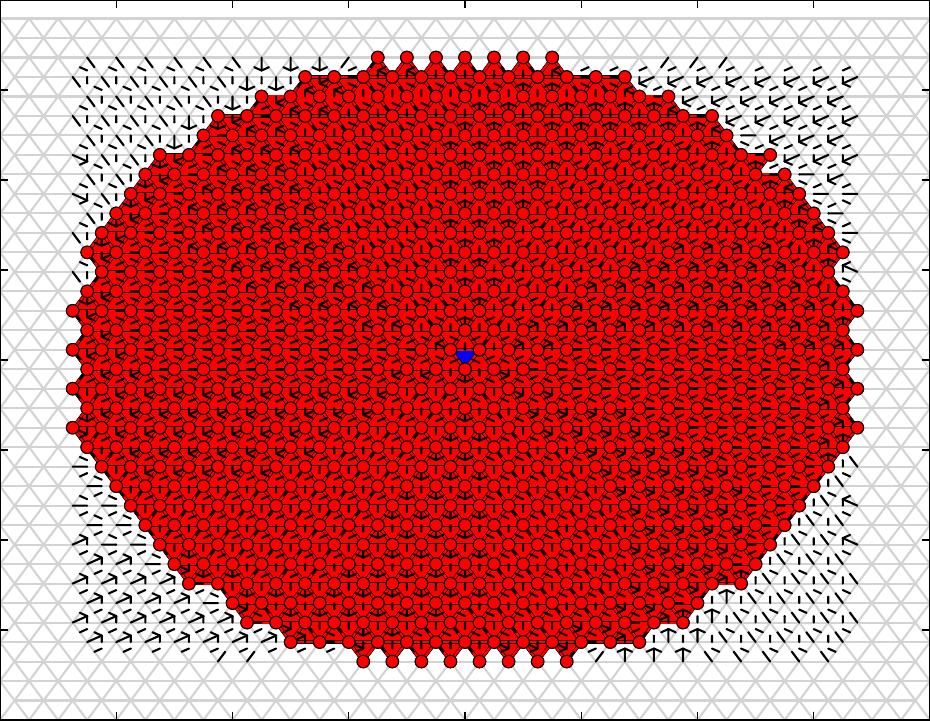}&
  \includegraphics[height=39.4mm]{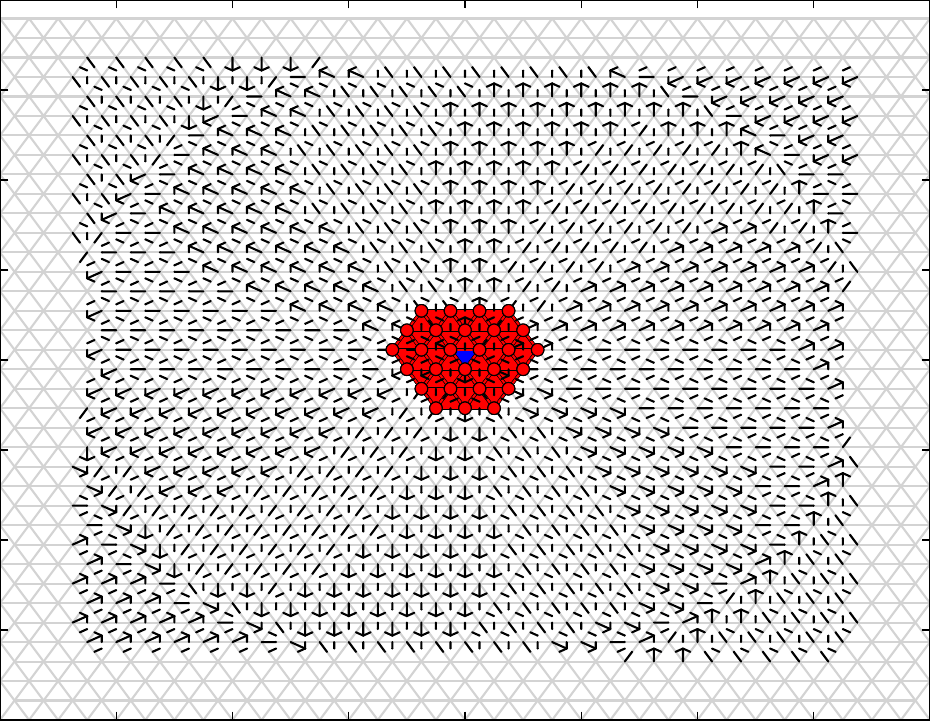}
\end{tabular}
\caption{Three different index pairs generated from our scheme in Algorithm \ref{alg:scheme}. The isolating neighborhood is in green, $E$ is in red, and $P \setminus E$ is in blue. Note how the isolating neighborhood changes by defining a collar around the invariant sets (which are exactly equal to $P \setminus E$). Between the left and middle multivector fields, the periodic attractor partially leaves $K$, so the maximal invariant set in $N$ is reduced to just a triangle. Hence, the size of $N$ drastically shrinks between the middle and right multivector fields.}
\label{fig:collar}
\end{figure} 

\section{Conclusion}
\label{set:conc}
In this paper, we focused on computing the persistence of Conley indices of isolated invariant sets. Our preliminary experiments show that the algorithm can effectively compute this persistence in the presence of noise. It will be interesting to derive
a stability theory for this persistence. Toward that direction, the stability result for the 
isolated invariant sets presented in Appendix~\ref{Stab-sec} seems promising. In designing the tracking algorithm, we have made certain choices about the isolated neighborhoods and the invariant sets. Are there better choices? Which ones work better in practice? A thorough investigation with data sets in practice is perhaps necessary to settle this issue.

\bibliographystyle{abbrv}
\bibliography{refs}

\begin{thebibliography}{10}

\bibitem{Ba2011}
J.~A. Barmak.
\newblock {\em Algebraic Topology of Finite Topological Spaces and
  Applications}.
\newblock Lecture Notes in Mathematics {\bf 2032}. Springer Verlag, Berlin -
  Heidelberg - New York, 2011.

\bibitem{Barmak2019}
J.~A. Barmak, M.~Mrozek, and T.~Wanner.
\newblock A {L}efschetz fixed point theorem for multivalued maps of finite
  spaces.
\newblock {\em Mathematische Zeitschrift}, May 2019.

\bibitem{BhSz1967}
N.~P. Bhatia and G.~P. Szeg\"{o}.
\newblock {\em Dynamical Systems: Stability Theory and Applications}.
\newblock Lecture Notes in Mathematics {\bf 35}. Springer Verlag, Berlin -
  Heidelberg - New York, 1967.

\bibitem{zigzag}
G.~Carlsson and V.~de~Silva.
\newblock Zigzag persistence.
\newblock {\em Foundations of Computational Mathematics}, 10(4):367--405, Aug
  2010.

\bibitem{Co78}
C.~Conley.
\newblock Isolated invariant sets and the {M}orse index.
\newblock In {\em CBMS Regional Conference Series 38, American Mathematical
  Society}, 1978.

\bibitem{DJKKLM19}
T.~K. Dey, M.~Juda, T.~Kapela, J.~Kubica, M.~Lipinski, and M.~Mrozek.
\newblock Persistent homology of {M}orse decompositions in combinatorial
  dynamics.
\newblock {\em {SIAM} J. Applied Dynamical Systems}, 18(1):510--530, 2019.

\bibitem{Forman1998b}
R.~Forman.
\newblock Combinatorial vector fields and dynamical systems.
\newblock {\em Math. Z.}, 228:629--681, 1998.

\bibitem{Forman1998a}
R.~Forman.
\newblock Morse theory for cell complexes.
\newblock {\em Adv. Math.}, 134:90--145, 1998.

\bibitem{HaleKocak1991}
J.~Hale and H.~Koçak.
\newblock {\em Dynamics and Bifurcations}.
\newblock Texts in Applied Mathematics {\bf 3}. Springer-Verlag, 1991.

\bibitem{hatcher}
A.~Hatcher.
\newblock {\em Algebraic topology}.
\newblock Cambridge University Press, Cambridge, 2002.

\bibitem{LKMW19}
M.~Lipi\'nski, J.~Kubica, M.~Mrozek, and T.~Wanner.
\newblock Conley-{M}orse-{F}orman theory for generalized combinatorial
  multivector fields on finite topological spaces.
\newblock arXiv:1911.12698 [math.DS], 2019.

\bibitem{Lo1963}
E.~N. Lorenz.
\newblock {\em Deterministic Nonperiodic Flow}, pages 25--36.
\newblock Springer New York, New York, NY, 2004.

\bibitem{MiMr1995}
K.~Mischaikow and M.~Mrozek.
\newblock Chaos in the {L}orenz equations: a computer-assisted proof.
\newblock {\em Bull. AMS (N.S.)}, 33:66--72, 1995.

\bibitem{MiMr2002}
K.~Mischaikow and M.~Mrozek.
\newblock {\em The Conley Index}.
\newblock Handbook of Dynamical Systems II: Towards Applications. (B. Fiedler,
  ed.) North-Holland, 2002.

\bibitem{MMRS1999}
K.~Mischaikow, M.~Mrozek, J.~Reiss, and A.~Szymczak.
\newblock Construction of symbolic dynamics from experimental time series.
\newblock {\em Phys. Rev. Lett.}, 82:1144--1147, Feb 1999.

\bibitem{Mr2017}
M.~Mrozek.
\newblock Conley--{M}orse--{F}orman theory for combinatorial multivector fields
  on {L}efschetz complexes.
\newblock {\em Foundations of Computational Mathematics}, 17(6):1585--1633, Dec
  2017.

\bibitem{munkres}
J.~Munkres.
\newblock {\em Topology}.
\newblock Featured Titles for Topology Series. Prentice Hall, Incorporated,
  2000.

\bibitem{Poinc1890}
H.~Poincar\'{e}.
\newblock Sur le probleme des trois corps et les \'equations de la dynamique.
\newblock {\em Acta Mathematica}, 13:1--270, 1890.

\end{thebibliography}

\appendix{

\section{Stability of invariant sets}
\label{Stab-sec}
We now establish a result on the stability of invariant sets under perturbations to the underlying multivector field. Note that since multivector fields are discrete, metrics on multivector fields over a given simplicial complex must likewise be discrete. Hence, rather than defining a metric on the space of multivector fields for some simplicial complex $K$, we consider a topology on the space of multivalued maps induced by multivector fields on $K$. To define the topology on this space, we first define a relation on the set of multivector fields. In particular, if $\mathcal{V}_1$ and $\mathcal{V}_2$ are multivector fields on simplicial complex $K$ with the property that for all $V_1 \in \mathcal{V}_1$ there exists a $V_2 \in \mathcal{V}_2$ where $V_1 \subset V_2$. We say that $V_1$ is \textit{inscribed} in $V_2$ and $\mathcal{V}_1$ is a \textit{refinement} of $\mathcal{V}_2$. We follow the notation in~\cite{DJKKLM19} and denote this relation as $\mathcal{V}_1 \sqsubset \mathcal{V}_2$. 

Unfortunately, there is no clear notion of stability between invariant sets under multivector field refinement. This is because if $\mathcal{V}_1 \sqsubset \mathcal{V}_2$, there can exist $\sigma \in K$ such that $\left[\sigma\right]_{\mathcal{V}_1}$ is critical while $\left[\sigma\right]_{\mathcal{V}_2}$ is not, or vice versa. Hence, we consider the notion of a \textit{strong refinement}. The multivector field $\mathcal{V}_1$ is a strong refinement of $\mathcal{V}_2$ if $\mathcal{V}_1$ is a refinement of $\mathcal{V}_2$ and for each regular $V_2 \in \mathcal{V}_2$, we have that $\esol_{\mathcal{V}_1}(V_2) = \emptyset$. We use the symbol $\mathcal{V}_1 \overline{\sqsubset} \mathcal{V}_2$ to denote the strong refinement relation. In addition, if $\mathcal{V}_1 \overline{\sqsubset} \mathcal{V}_2$, then we write that $F_{\mathcal{V}_1} \overline{\subset} F_{\mathcal{V}_2}$. This notation is motivated by the fact that the definition of $F_\mathcal{V}$ established in Section \ref{sec:prelim} implies that $F_{\mathcal{V}_1}(\sigma) \subset F_{\mathcal{V}_2}(\sigma)$ for all $\sigma \in K$. 

We use the relation $\overline{\subset}$ to define a partial order on the space of dynamics \[F_K=\{F_{\mathcal V}\,|\,{\mathcal V} \mbox{ is on } K\}.\] In particular, if $F_{\mathcal{V}_1} \overline{\subset} F_{\mathcal{V}_2}$, we write $F_{\mathcal{V}_2} \leq_F F_{\mathcal{V}_1}$. 
\begin{proposition}
The relation $\leq_F$ is a partial order.
\end{proposition}
\begin{proof}
Note that for every multivector field $\mathcal{V} \overline{\sqsubset} \mathcal{V}$, so it follows that for all $F_\mathcal{V} \in F_K$, we have that $F_\mathcal{V} \leq_F F_\mathcal{V}$. Hence, $\leq_F$ is reflexive. Similarly, if $\mathcal{V}_1 \overline{\sqsubset} \mathcal{V}_2$ and $\mathcal{V}_2 \overline{\sqsubset} \mathcal{V}_1$, then it follows that for $V_1 \in \mathcal{V}_1$ there exists a $V_2 \in \mathcal{V}_2$ such that $V_1 \subset V_2$. But likewise, there must exist a $V'_1 \in \mathcal{V}_1$ such that $V_2 \subset V'_1$. But a multivector is a partition, and $V_1 \subset V_2 \subset V'_1$. Hence, $V'_1 \cap V_1 \neq \emptyset$, so $V'_1 = V_1$. Thus, $V_1 \subset V_2$ and $V_2 \subset V_1$, so $\mathcal{V}_1 = \mathcal{V}_2$ which implies that $F_{\mathcal{V}_1} = F_{\mathcal{V}_2}$. Thus, $\leq_F$ is antisymmetric.

Finally, we show that $\leq_F$ is transitive. Note that if $F_{\mathcal{V}_1} \leq F_{\mathcal{V}_2}$ and $F_{\mathcal{V}_2} \leq F_{\mathcal{V}_3}$, we have that $F_{\mathcal{V}_3} \overline{\subset} F_{\mathcal{V}_2}$ and that $F_{\mathcal{V}_2} \overline{\subset} F_{\mathcal{V}_1}$. Hence, since $\mathcal{V}_1, \mathcal{V}_2, \mathcal{V}_3$ are all defined on the same $K$, we have that for $V_1 \in \mathcal{V}_1$, there exists a $V_2 \in \mathcal{V}_2$ such that $V_1 \subset V_2$. Similarly, there exists a $V_3 \in \mathcal{V}_3$ such that $V_2 \subset V_3$. Hence, $\mathcal{V}_1$ is a refinement of $\mathcal{V}_3$, but it is not obvious that this is a strong refinement. Note that $\mathcal{V}_2$ is a strong refinement of $\mathcal{V}_3$, so for every $V_3 \in \mathcal{V}_3$, we have that $\esol_{\mathcal{V}_2}(V_3) = \emptyset$. We aim to show that $\esol_{\mathcal{V}_1}(V_3) = \emptyset$. Aiming for a contradiction, assume that $\esol_{\mathcal{V}_1}(V_3) \neq \emptyset$. In particular, let $\rho \in \esol_{\mathcal{V}_1}(V_3)$ denote such an invariant set. Note that $\rho(\mathbb{Z}) \not\subset V_2$, where $V_2 \subset V_3$. If the image of $\rho$ were in some $V_2$, then this contradicts $\mathcal{V}_1$ being a strong refinement of $\mathcal{V}_2$. In addition, there cannot exist an $a$ such that $\rho\left(\mathbb{Z}_{[-\infty, a]} \right) \subset V_2 $, as this also contradicts $\mathcal{V}_1$ being a strong refinement. There likewise can't be an $a$ so that $\rho\left( \mathbb{Z}_{[a, \infty]} \right) \subset V_2 $. Hence, for all $i$, there exists $j,k$ satisfying $j < i < k$ and $\left[\rho(j)\right]_{\mathcal{V}_2} \neq \left[\rho(i)\right]_{\mathcal{V}_2} \neq \left[\rho(k)\right]_{\mathcal{V}_2}$. This implies that $\rho \in \esol_{\mathcal{V}_2}(V_3)$, a contradiction. 
\end{proof}

Given a poset $(X, \leq_X)$, a set $U \subset X$ is said to be \textit{upper} if for all $x \in X, u \in U$ where $u \leq x$, we have that $x \in U$. \textit{Lower} sets are defined as is expected. It is well known that upper sets induce a topology. 
\begin{definition}
The {\em Alexenadrov topology} on a poset $(X,\leq_X)$ is the topology on $X$ in which all upper sets with respect to $\leq_X$ are open. 
\end{definition}
In addition, it is well known that lower sets correspond to closed sets. Alexenandrov topologies have several convenient properties not found in arbitrary topological spaces. Notably, in the Alexandrov topology on $X$, every point in $X$ has a minimal neighborhood.

Throughout the remainder of this section, we consider the Alexandrov topology on $(F_K,\leq_F)$. In particular, we show that invariant sets are \textit{strongly upper semicontinuous} with respect to this topology. Since we are only dealing with finite spaces, we use a definition from \cite{Barmak2019}. 
\begin{definition}[Strongly Upper Semicontinuous]
Let $F \; : \; X \multimap Y$ denote a multivalued map from a $T_0$ topological space $X$ to some set $Y$. The map $F$ is {\em strongly upper semicontinuous} if for all $x_1, x_2 \in X$ where $x_1 \leq x_2$, $F(x_1) \subseteq F(x_2)$. 
\end{definition}
Let $\inv = \{ \inv(F_\mathcal{V}, N) \; | \; F_\mathcal{V} \in F_K \}$ denote the set of invariant sets induced by multivector fields over $K$ in some fixed $N$. As each $F_\mathcal{V}$ defines an invariant set, we consider the map $p \; : \; F_K \to \inv$ which maps each $F_\mathcal{V}$ to its invariant part in $N$. We conclude this subsection by showing that $p$ is strongly upper semicontinous. 
\begin{proposition}
Let $F_{\mathcal{V}_1} \leq_F F_{\mathcal{V}_2}$ denote multivalued maps induced by multivector fields $\mathcal{V}_2 \overline{\sqsubset} \mathcal{V}_1 $ on some simplicial complex $K$. For any closed $N \subset K$, $\inv(F_{\mathcal{V}_2},N) \subset \inv(F_{\mathcal{V}_1}, N)$. 
\end{proposition}
\begin{proof}
For contradiction, assume that there exists a solution $\rho \; : \; \mathbb{Z} \to N$ where $\rho \in \inv(F_{\mathcal{V}_2},N)$ but $\rho \not\in \inv(F_{\mathcal{V}_1}, N)$. This implies that $\rho$ is not essential with respect to $\mathcal{V}_1$, so there must exist some $a \in \mathbb{Z}, V \in \mathcal{V}_1$, $V$ regular such that $\rho\left( \left[a,\infty\right) \right) \subset V$ or $\rho\left( \left(-\infty, -a\right] \right) \subset V$. But this implies that $\esol_{\mathcal{V}_1}(V) \neq 0$, which contradicts $\mathcal{V}_2$ being a strong refinement of $\mathcal{V}_1$. 
\end{proof}

\begin{corollary}
The map $p$ is strongly upper semicontinuous. 
\end{corollary}

}

\end{document}